\documentclass[11pt, english,english]{article}
\usepackage{amsmath,amssymb,amsthm} 
\usepackage{bbm}

\usepackage{mathrsfs,bm,enumerate}
\usepackage{graphicx,color,tikz}
\usepackage{caption,subcaption}
\usepackage[T1]{fontenc} 
\usepackage[latin9]{inputenc}
\usepackage{amsthm}
\usepackage{mathtools}
\usepackage{bookmark}

\newcommand{\inR}{\in \mathbb{R}}

\newcommand{\R}{ \mathbb{R}}
\newcommand{\Z}{ \mathbb{Z}}

\newcommand{\N}{ \mathbb{N}}

\newcommand{\eqdef}{\stackrel{\vartriangle}{=}}

\newcommand{\Lop}{{\rm L}}
\newcommand{\Dop}{{\rm D}}

\newcommand{\dint}{{\rm d}}

\newcommand{\Fourier}{ \mathcal{F}}

\newcommand{\bw}{{\boldsymbol \omega}}

\DeclareMathOperator*{\esssup}{ess\,sup}

\def\V#1{{\boldsymbol{#1}}}         
\def\Spc#1{{\mathcal{#1}}}  
\def\M#1{{\bf{#1}}}  
\def\Op#1{{\mathrm{#1}}}  
\def\ee{\mathrm{e}}

\def\Proj{\mathrm{Proj}} 
\def\Identity{\mathrm{Id}} %

\newcommand{\embedC}{\xhookrightarrow{}}
\newcommand{\embedD}{\xhookrightarrow{\mbox{\tiny \rm d.}}}
\newcommand{\embedIso}{\xhookrightarrow{\mbox{\tiny \rm iso.}}}
\newcommand{\toC}{\xrightarrow{\mbox{\tiny \ \rm c.  }}}
\newcommand{\toIso}{\xrightarrow{\mbox{\tiny \ \rm iso.  }}}
\def\ie{{\em i.e.}, }
\def\eg{{\em e.g.}, }

\renewcommand{\[}{\begin{equation}}
\renewcommand{\]}[1]{\label{eq:#1}\end{equation}}

\newtheorem{definition}{Definition}
\newtheorem{proposition}{Proposition}
\newtheorem{corollary}{Corollary}

\newtheorem{theorem}{Theorem}
\newtheorem{example}{Example}

\title{Native Banach spaces for splines\\ and 
variational inverse problems\thanks{The research leading to these results has received funding from the Swiss National Science Foundation under Grant 200020-162343.}}

\author{
Michael Unser and Julien Fageot\thanks{Biomedical Imaging Group, \'Ecole polytechnique f\'ed\'erale de Lausanne (EPFL),
Station 17, CH-1015, Lausanne, Switzerland ({\tt michael.unser@epfl.ch}). }
 }

\begin{document}



\maketitle
\begin{abstract}
We propose a systematic construction of native Banach spaces for general spline-admissible
operators $\Lop$. In short, the native space for $\Lop$ and the (dual) norm $\|\cdot\|_{\Spc X'}$ is the largest space of functions $f: \R^d \to \R$ such that $\|\Lop f\|_{\Spc X'}<\infty$, subject to the constraint that the growth-restricted null space of $\Lop$ 
be finite-dimensional. This space, denoted by $\Spc X'_\Lop$, is specified as the dual
of the pre-native space $\Spc X_\Lop$, which is itself obtained through a suitable completion process. The main difference with prior constructions (\eg reproducing kernel Hilbert spaces) is that our 
approach involves test functions rather than sums of atoms (\eg kernels), which makes it applicable to a much broader class of norms, including total variation. Under specific admissibility and compatibility hypotheses,
we  lay out the direct-sum topology of $\Spc X_\Lop$ and $\Spc X'_\Lop$, and
identify the whole family of equivalent norms. Our construction ensures that the native space
and its pre-dual are endowed with a fundamental Schwartz-Banach property. In practical terms, this means that $\Spc X'_\Lop$ 
is rich enough to reproduce any function 
with an arbitrary degree of precision.

\end{abstract}

\tableofcontents
%

\section{Introduction}
Given a series of data points $(\V x_m,y_m) \in \R^d \times \R$, the basic interpolation problem is to find a function
$f: \R^d \to \R$ such that $f(\V x_m)=y_m$ for $m=1,\dots,M$. In order to make the problem well posed, one needs
to impose additional constraints on $f$; for instance that $f$ is in the linear span of a finite number of known basis functions
(standard regression problem), or that the desired function minimizes some energy functional.
The variational formulation of splines follows the latter strategy and ensures the existence and unicity of the solution
for energy functionals of the form $\|\Lop f\|^2_{L_2}$ with $\Lop$ a suitable differential operator, the prototypical example being $\Lop=\Dop^n$ with $\Dop$ the derivative operator \cite{Schoenberg1964,Prenter:1975,deBoor1966}.

We adopt in this paper an abstract formulation that 
encompasses the great majority of variational theories of splines that have been considered in the literature such as
\cite{DeBoor1976, Bezhaev2001} to give some notable examples.
%
%
Given a suitable\footnote{Our notation reflects the property that the two spaces are the topological duals of two primary spaces $(\Spc X, \Spc X_\Lop)$. $\Spc X$ is a classical function space such as $L_p(\R^d)$, while $\Spc X'_\Lop$ is the corresponding native space for $\Lop$; that is, the largest Banach space such that $\|\Lop f\|_{\Spc X'}$ is well-defined. } pair of Banach (or Hilbert) spaces $(\Spc X',\Spc X'_\Lop)$ whose elements are functions on $\R^d$, a regularization operator 
$\Lop: \Spc X_\Lop' \toC \Spc X'$, and a vector-valued measurement functional $\V \nu 
: \Spc X_\Lop' \to \R^M,
f\mapsto \V \nu(f)=(\langle \nu_1,f\rangle,\dots,\langle \nu_M,f\rangle)$, we define the generalized spline interpolant $f_{\rm int}: \R^d \to \R$ as the solution of the variational linear inverse problem
\begin{align}
\label{Eq:VariationalSpline}
f_{\rm int}=\arg \min_{f \in \Spc X'_\Lop} \|\Lop f\|_{\Spc X'} \quad \mbox{ s.t. } \quad \V \nu(f)=(y_1,\dots,y_M).
\end{align}
The simplest case occurs when there is an isomorphism between $\Spc X'_\Lop$ and $\Spc X'$, meaning that the regularization operator $\Lop$
has a stable inverse $\Lop^{-1}: \Spc X' \toC \Spc X_\Lop'$.
In particular, when $\nu_m=\delta(\cdot-\V x_m): f \mapsto f(\V x_m)$ and $\Spc H=\Spc X'_\Lop$ is a reproducing-kernel Hilbert space (RKHS) such that 
$\langle f, g \rangle_{\Spc H}=\langle \Lop f, \Lop g \rangle=\langle \Lop^\ast\Lop f,  g \rangle$
for all $f,g \in \Spc H$, then the solution of
\eqref{Eq:VariationalSpline} with $\Spc X'=L_2(\R^d)$ is expressible as\begin{align}
\label{Eq:kernelExp}
f_{\rm int}=\sum_{m=1}^M a_m h(\cdot,\V x_m),
\end{align}
where 
\begin{align}
\label{Eq:kernelFormula}
h(\cdot,\V x_m)=(\Lop^\ast\Lop)^{-1}\{\delta(\cdot-\V x_m)\}
\end{align} 
is the (unique) reproducing kernel of $\Spc H$  \cite{deBoor1966,Berlinet2004}. The bottom line is that \eqref{Eq:kernelExp} is a linear model parametrized by $\V a
\inR^M$.
In addition, we have that $\|f_{\rm int}\|^2_{\Spc H}=\V a^T \M G \V a$, where $\M G\inR^{M \times M}$ is the positive-definite Gram matrix whose entries are given by
$[\M G]_{m,n}=h(\V x_m,\V x_n)$; this then yields the solution $\M a=\M G^{-1} \M y$ of our initial interpolation problem with $\M y=(y_1, \dots,y_M)$.

The theory of RHKS also works the other way around in the sense that,  instead of the operator $\Lop$, one can specify a positive-definite kernel
$h: \R^d \times \R^d \to \R$ \cite{Aronszajn1950}. One then constructs the native space $\Spc X'_\Lop=\Spc H$
by considering the closure of the vector space specified by \eqref{Eq:kernelExp} with varying $M\in \N$ and $\V x_m\inR^d$; \ie
 $\Spc H=\overline{{\rm span}}\{h(\cdot,\V x)\}_{\V x\inR^d}$ \cite{Berlinet2004,Ferreira2013}.

This kind of result is extendable to the scenario where $\Lop$ has a finite-dimensional null space $\Spc N_\Lop={\rm span}\{p_n\}_{n=1}^{N_0}$, in which case the solution takes the generic parametric form 
\begin{align*}
f_{\rm int}=\sum_{m=1}^M a_m h(\cdot,\V x_m) + \sum_{n=1}^{N_0} {b_n} p_n
\end{align*}
with expansion coefficients $\V a=(a_1,\dots,a_M)\inR^M$ and $\V b=(b_1,\dots,b_{N_0})$ \cite{deBoor1966,Wahba1990}. The delicate point there is to properly define
the corresponding native space $\Spc H$ since the underlying kernel $h: \R^d \times \R^d \to \R$, which is still given by \eqref{Eq:kernelFormula}, is no longer positive-definite \cite{Micchelli1986, Schaback1999,Schaback2000}. The underlying native space then has the structure of a semi-RKHS  \cite{Mosamam2010, Atteia1992}, a concept that was already present implicitly in the early works on variational splines \cite{deBoor1966, Duchon1977}.

While the aforementioned results with $\Spc X'=L_2(\R^d)$ are classical, there has been recent interest in a variant of Problem \eqref{Eq:VariationalSpline}
with $\Spc X'=\Spc M(\R^d)$ (the space of Radon measures on $\R^d$) \cite{Unser2017,flinth2017,Bredies2018}. It turns out that the latter is much better suited to compressed sensing and to the resolution of inverse problems in general \cite{Gupta2018}. In particular, when $\Lop$ is shift-invariant, it has been shown \cite[Theorem  2]{Unser2017} that
the minimum of $\|\Lop f\|_{\Spc M}$ is achieved by an adaptive $\Lop$-spline\footnote{A function $f: \R^d\to \R$ is said to be an $\Lop$-spline with knots $\V \tau_1,\dots,\V \tau_K \inR^d$ if
$\Lop\{f\}=\sum_{k=1}^K a_k \delta(\cdot-\V \tau_k)$
with $a_1,\dots,a_K\inR$. For instance, if $d=1$ and $\Lop=\Dop$ (resp. $\Lop=\Dop^2$), then $f$ is piecewise-constant (resp. piecewise-linear and continuous)
with jumps (resp. breaks) at the $\tau_m$.}
of the form
\begin{align}
\label{Eq:Sparse}
f=\sum_{k=1}^K a_k \rho_\Lop(\cdot-\V \tau_k) + \sum_{n=1}^{N_0} {b_n} p_n,
\end{align}
where $\rho_\Lop=\Lop^{-1}\{\delta\}$, 
with the twist that the intrinsic parameters of the spline---the number $K$ of knots and their locations $\V \tau_k$---are adjustable with $K\le (M-N_0)$. Remarkably, the generic form \eqref{Eq:Sparse} of the solution is preserved for arbitrary continuous linear measurement functionals $\V \nu: \Spc X_\Lop' \to \R^M$, far beyond the pure sampling framework of RKHS, which makes the result applicable to linear inverse problems. While there have been attempts to generalize the $\Spc M$-norm (or total variation) variant of the reconstruction problem \eqref{Eq:VariationalSpline} \cite{flinth2017,Bredies2018,Boyer2018}, the part that has remained elusive
is the specification of the corresponding native space in the multidimensional scenario ($d>1$) when the null space of $\Lop$ is nontrivial. The difficulty there is essentially the same as the one that was encountered initially by Duchon with $\Lop=(-\Delta)^{\gamma}$ (fractional Laplacian) and $\Spc X=L_2(\R^d)$ \cite{Duchon1977}; namely, the need to properly restrict the native space in order to exclude the harmonic components of the null space that grow faster than the underlying Green's function. A systematic approach for specifying native spaces was later given by Schaback \cite{Schaback1999,Schaback2000}, but it is restricted to the RKHS framework and to kernels that are conditionally positive-definite.

In this paper, we develop an alternative Banach-space formulation that is applicable to a whole range of norms $\|\cdot\|_{\Spc X'}$
and operators $\Lop$. In a nutshell, we are proposing a systematic approach for constructing the largest Banach space $\Spc X'_\Lop$ that ensures that
$\|\Lop f\|_{\Spc X'}$ is well defined, subject to the constraint that the null space of $\Lop$ should be finite-dimensional.
The main benefits of our new formulation
are as follows:

\begin{itemize}
\item The extension of the notion of generalized spline via \eqref{Eq:VariationalSpline} 
through the appropriate pairing of a regularization operator $\Lop$ and a space $\Spc X'$ that is the continuous dual of a primary Banach space $\Spc X$. The two aforementioned theories with $R(f)=\|\Lop f\|^2_{L_2}$ (RKHS or Tikhonov regularization) and  $R(f)=\|\Lop f\|_{\Spc M}$ (generalized TV regulatization) are covered by taking
$\Spc X=L_2(\R^d)$ and $\Spc X=C_0(\R^d)$ 
\big(the pre-dual of $\Spc M(\R^d)$\big), but our framework is considerably more general.

\item The precise identification of the class of spline-admissible operators $\Lop$ (see Definition \ref{Def:splineadmis3}). 
In essence, these are differential-like operators that are injective on $\Spc S(\R^d)$ (Schwartz' class of test functions) and that admit a well-defined Green's function, which is symbolically denoted by $(\V x, \V y) \mapsto g_\Lop(\V x,\V y)=\Lop^{-1}\{\delta(\cdot-\V y)\}(\V x)$, where $\delta(\cdot-\V y)$ is the Dirac impulse at the location $\V y\inR^d$.

\item The proposal of a Banach counterpart to the notion of conditional positive definiteness from the theory of semi-RKHS: The critical hypothesis here is the continuity of some 
pseudo-inverse operator $\Lop^{-1\ast}_\V \phi$ (see Definition \ref{Def:Hypotheses}). In a companion paper, we shall demonstrate the usefulness of this criterion and how it can be readily tested in practice \cite{Fageot2019a}.

\item
A systematic way of constructing the native space for $(\Lop,\Spc X')$, denoted by $\Spc X'_\Op L$, via a completion process that involves test functions and operators rather than
linear combinations of kernels.  This is a significant extension of the usual approach as it applies to a much larger family of primary spaces, including non-reflexive Banach spaces such as $\Spc M(\R^d)$.

\item The guarantee of universality: The native space
$\Spc X'_\Op L$ specified by Theorem \ref{Theo:NativeSpace2} is rich enough to represent any function with an arbitrary degree of precision\footnote{This is due to $\Spc X'_\Op L$ being dense in $\Spc S'(\R^d)$. 
}. It is also sufficiently restrictive for the null space of $\Lop$ to be finite-dimensional, which is non-obvious for $d>1$ since the ``unrestricted'' null space 
of partial differential operators is either trivial or infinite-dimensional \cite{Hormander1990}. The proposed approach is a convenient alternative to the more traditional use of Beppo-Levi spaces that
involve composite norms with partial derivatives \cite{Beatson2005,Deny1954, Berlinet2004,Wendland2005}.

\item The explicit characterization of $\Spc X_\Lop$ (the pre-dual of the native space $\Spc X'_\Lop$), as given 
in Corollary \ref{Corol:pre-dual}.
The practical significance of this result 
is that it precisely delineates the domain of validity of representer theorems for the solutions of Problem \eqref{Eq:VariationalSpline} or variants thereof.
Indeed, a sufficient 
condition for existence is the weak* continuity of the measurement functional $\V \nu: \Spc X'_\Lop \to \R^M$, as in \cite{Unser2017}. This is equivalent to $\nu_m \in \Spc X_\Lop$ for $m=1,\dots,M$ \cite{Reed1980}. For instance, the requirement for the well-possedness
of a regularized interpolation problem is $\delta(\cdot-\V x_m)\in \Spc X_\Lop$ for any $\V x_m \in \R^d$. The latter condition is automatically satisfied when $\Spc X'_\Lop$ is a RKHS. However, it can fail when switching to
non-euclidean norms such as, for example, the critical configuration $(\Lop,\Spc X')=(\Dop,\Spc M(\R))$ which corresponds to the popular total-variation regularization with
$R(f)={\rm TV}(f)=\|\Dop f\|_{\Spc M}$ \cite{Rudin1992,Chambolle2004a}. On the other hand, we have that
$\delta(\cdot-\V x_m)\in \Spc X_\Lop$ for $(\Lop,\Spc X')=(\Dop^2,\Spc M(\R))$, which justifies the use of second-order total variation for the regularization of deep neural networks \cite{Unser2018}.

\end{itemize}

The paper is organized as follows: In Section \ref{Sec:Prelinaries}, we lay out the functional context while introducing the notion of Schwartz-Banach space, which is fundamental to our approach. In Section 3, we give the mathematical conditions for $\Lop$ to be spline-admissible and describe an effective way to stabilize its inverse via the use of a biorthogonal system of $\Spc N_\Lop$ (the null space of $\Lop$). The final ingredient is given by some norm-compatibility conditions 
that enable the specification of the pre-Banach space $\Spc P_\Lop$.
The completion of $\Spc P_\Lop$
in the $\|\cdot\|_{\Spc X_\Lop}$-norm yields our pre-native space $\Spc X_\Lop$
whose properties are revealed in the first part of Section 4 (Theorem \ref{Theo:pre-dual2}).
This then allows us to characterize the native space $\Spc X'_\Lop$ (Theorem \ref{Theo:NativeSpace2}) and to establish its embedding properties. We also identify the complete family of equivalent norms, which yields a complete understanding of the underlying direct-sum topology.
Finally, in Section 5, we 
illustrate the compatibility of our extended formulation with the specification of many classical spaces; in particular, RKHS and $L_p$-type Solobev spaces.
\section{Preliminaries}
\label{Sec:Prelinaries}

\subsection{Schwartz-Banach spaces}

The Banach spaces that we shall consider are embedded in Schwartz' space of tempered distributions denoted by $\Spc S'(\R^d)$. Formally, an element $f\in\Spc S'(\R^d)$
is a continuous linear functional $f: \varphi \mapsto \langle f, \varphi\rangle$ that associates a real number denoted by 
$\langle f, \varphi\rangle$ to each test function $\varphi \in \Spc S(\R^d)$ (Schwartz' space of smooth and rapidly decaying functions). 
For instance, the Dirac impulse at location $\V x_0\in\R^d$
is the generalized function $\delta(\cdot-\V x_0)\in \Spc S'(\R^d)$  defined as
$$
\varphi \mapsto \langle \delta(\cdot-\V x_0), \varphi\rangle=\varphi(\V x_0).
$$
Similarly, any slowly increasing and locally integrable function $f: \R^d \to \R$ specifies a distribution by way of the integral
$\langle f, \varphi\rangle=\int_{\R^d} f(\V x)\varphi(\V x)\dint \V x$.

Our construction relies on the pairing of an operator $\Lop$ and a Banach space $\Spc X'$.
The latter is the dual of a primary space $\Spc X$ that is appropriately sandwiched between $\Spc S(\R^d)$ and $\Spc S'(\R^d)$, in accordance with Definition \ref{Def:SchwartzBanach}.

\begin{definition}[Schwartz-Banach space]
\label{Def:SchwartzBanach}
A normed vector space $(\Spc X,\|\cdot\|_{\Spc X})\subseteq \Spc S'(\R^d)$ is said to be a Schwartz-Banach space
if it can be specified as the completion of $\Spc S(\R^d)$ in the $\|\cdot\|_{\Spc X}$-norm or, equivalently, if 
$\Spc X$ is a Banach space with the property that $\Spc S(\R^d) \embedD \Spc X \embedD \Spc S'(\R^d)$ (continuous and dense embeddings).
\end{definition} 

The reader is  referred to Appendix A for the precise definition of the underlying notions of embeddings and a review of supporting mathematical results.

Prominent examples of Schwartz-Banach spaces are $L_p(\R^d)$ with $p\in[1,\infty)$ and $C_0(\R^d)$ (the class of functions that are continuous, uniformly bounded and decaying at infinity).
By contrast, the Schwartz-Banach property holds neither for $L_\infty(\R^d)$ nor for $\Spc M(\R^d)$ (the space of bounded Radon measures on $\R^d)$. These two spaces, however, retain relevance for our purpose as duals of
(non-reflexive) Schwartz-Banach spaces; i.e., $L_\infty(\R^d)=\big(L_1(\R^d)\big)'$ and $\Spc M(\R^d)=\big(C_0(\R^d)\big)'$.
\begin{proposition}
\label{Prop:SchwartzBanach}
The dual $\Spc X'$ of a Schwartz-Banach space $\Spc X$ is a Banach space with the following properties:
\begin{itemize}
\item $\Spc S(\R^d) \embedC \Spc X' \embedD \Spc S'(\R^d)$
\item $\Spc X'=\{w \in \Spc S'(\R^d):  \sup_{\varphi \in \Spc S(\R^d)\backslash \{0\}} \frac{\langle w,\varphi\rangle}{\|\varphi\|_{\Spc X}}=\|w\|_{\Spc X'}<+\infty\}.$
\item The duality product for $(\Spc X', \Spc X)$ is compatible with that of $\big(\Spc S'(\R^d), \Spc S(\R^d)\big)$, which allows us to write
\begin{align}
\label{Eq:CompatibleDproduct}
\langle f, g\rangle_{\Spc X' \times \Spc X}=\langle f, g\rangle
\end{align}
for any $(f,g) \in \Spc X'\times \Spc X$. 
\end{itemize}
If, in addition, $\Spc X$ is reflexive, then $\Spc X'$ is itself a Schwartz-Banach space, which then also yields
$$
\Spc X=\Spc X''=\{g \in \Spc S'(\R^d): \|g\|_{\Spc X}=\sup_{\varphi \in \Spc S(\R^d)\backslash \{0\}} \frac{\langle g,\varphi\rangle}{\|\varphi\|_{\Spc X'}}<+\infty\}.
$$

\end{proposition}
The prototypical example that meets the last statement is $\Spc X'=L_q(\R^d)=\big(L_p(\R^d)\big)'$ with $\frac{1}{p} + \frac{1}{q}=1$ and $1<p<\infty$. The property that is lost when 
$\Spc X$ is not reflexive \big(\eg for $(p,q)=(1,\infty)$\big) is the denseness of the continuous embedding $\Spc S(\R^d)\embedC\Spc X'$.

\begin{proof} The embedding $\Spc S(\R^d) \embedC \Spc X' \embedD \Spc S'(\R^d)$ follows from
the Schwartz-Banach property, Theorem \ref{Theo:Dense} on dual embeddings, and the reflexivity of $\Spc S(\R^d)$. When $\Spc X$ is reflexive, we also get
that $\Spc S(\R^d) \embedD \Spc X'$, which proves that $\Spc X'$ is a Schwartz-Banach space.
By definition, 
$\Spc X'$ is the Banach space associated with the dual norm
\begin{align*}
\|w\|_{\Spc X'}\eqdef 
 \sup_{\|\varphi \|_\Spc X\le 1:\ \varphi \in \Spc X} \langle w,\varphi\rangle_{\Spc X' \times \Spc X}<+\infty.
\end{align*}
Now, the fundamental observation is that $\langle w,\varphi\rangle_{\Spc X' \times \Spc X}=\langle w,\varphi\rangle$
for any $(w,\varphi)\in \big(\Spc X', \Spc S(\R^d)\big)$ since $\Spc X' \subseteq \Spc S'(\R^d)$,
while the determination of the supremum 
can be restricted to $\varphi \in \Spc S(\R^d)$ in reason of the denseness of  $\Spc S(\R^d)$ in $\Spc X$. This allows us to rewrite the dual norm as\begin{align}
\label{Eq:dualNormSchwartz}
\|w\|_{\Spc X'}
=\sup_{\|\varphi\|_{\Spc X}\l\le 1: \ \varphi \in \Spc S(\R^d)} \langle w,\varphi\rangle=\sup_{\varphi \in \Spc S(\R^d)\backslash \{0\}} \frac{\langle w,\varphi\rangle}{\|\varphi\|_{\Spc X}}.
\end{align}
%
Since $\Spc X'\embedC \Spc S'(\R^d)$ and the expression on the right-hand side of \eqref{Eq:dualNormSchwartz} is well-defined for any $w \in \Spc S'(\R^d)$, we have that
$\Spc X' \subseteq \Spc W=\{w \in \Spc S'(\R^d): \|w\|_{\Spc X'}<\infty\}$.
To establish the converse inclusion---and, hence, $\Spc X'=\Spc W$---we use
\eqref{Eq:dualNormSchwartz} 
to infer that, for any $w \in \Spc W\subseteq\Spc S'(\R^d)$, the linear functional $w: \Spc S(\R^d) \to \R$ is 
bounded by
$$
 |\langle w, \varphi\rangle|\le \|w\|_{\Spc X'}\|\varphi\|_{\Spc X}
$$
for all $\varphi \in \Spc S(\R^d)$. We then invoke Theorem \ref{Theo:BLT} below with $\Spc Y=\R$ to deduce that
$w\in \Spc W$ has a unique continuous extension to the completed space
$\Spc X=\overline{(\Spc S(\R^d),\|\cdot\|_{\Spc X})}$, which proves that
$\Spc W \subseteq \Spc X'$. The same argument also gives a precise meaning
to the right-hand side of \eqref{Eq:CompatibleDproduct}.
%
Finally, in the reflexive case, we reapply the first part of Proposition \ref{Prop:SchwartzBanach} to $\Spc X'$ to obtain the announced characterization of $\Spc X''=\Spc X$.
\end{proof}

The important point here is that the two (dual) formulas for the norms in Proposition \ref{Prop:SchwartzBanach} are valid for any tempered distribution $w \in \Spc S'(\R^d)$. In effect, this provides us with a simple criterion for space membership (resp. exclusion):
$ \|w\|_{\Spc X'}<\infty \Leftrightarrow w \in \Spc X'$ (resp. $w \notin \Spc X' \Leftrightarrow \|w\|_{\Spc X'}=\infty$).
Likewise, we have that
$g \in \Spc X=\Spc X'' \Leftrightarrow \|g\|_{\Spc X}<\infty$, although the equivalence holds true in the reflexive case only. This explains the 
greater difficulty in developing a general theory for non-reflexive native spaces\footnote{For instance, the non-reflexive spaces $C_0(\R^d)\embedIso C_{\rm b}(\R^d)\embedIso L_\infty(\R^d)$ are all equipped with the same $\|\cdot\|_\infty$-norm.
The condition that $g$ is bounded 
is not sufficient to ensure that $g \in C_0(\R^d)$ (the only Schwartz-Banach space in the chain); it is also required that $g$ be continuous 
and decaying at infinity. The boundedness criterion for space membership applies to $L_\infty(\R^d)$ alone because it is the dual of the Schwartz-Banach space $L_1(\R^d)$.}---as opposed to, say, the 
traditional RKHS that go hand-in-hand with $\Spc X=L_2(\R^d)$.

A significant advantage of working with Schwartz-Banach spaces is the 
compatibility property expressed by \eqref{Eq:CompatibleDproduct}.
In addition, when $f \in \Spc X'$ and $g \in \Spc X$ are both ordinary functions, the duality product has an explicit transcription as the integral
\begin{align}
\label{Eq:dualproduct}
\langle f, g\rangle=\int_{\R^d} f(\V x)g(\V x) \dint \V x.
\end{align}
In order to control the algebraic rate of growth/decay of such functions, we 
rely 
on weighted Banach spaces such as
$$
L_{\infty,\alpha}(\R^d)=\left\{f: \R^d \to \R \ \mbox{measurable  s.t. }  \|f\|_{\infty,\alpha}<+\infty\right\}
$$
with norm
$$
\|f\|_{\infty,\alpha}\eqdef \|(1+\| \cdot\|)^{-\alpha} f\|_{\infty}=\esssup_{\V x \inR^d} \left((1+\|\V x\|)^{-\alpha} |f(\V x)|\right),
$$
where the order $\alpha\in \R^+$ puts a cap on the rate of growth of $f$ at infinity.
In conformity with the relation
$L_\infty(\R^d)=\big( L_1(\R^d)\big)'$,  $L_{\infty,\alpha}(\R^d)$ is the continuous dual of $L_{1,-\alpha}(\R^d)$:
the Schwartz-Banach space associated with the weighted $L_1$ norm
$$
\|g\|_{1,-\alpha}\eqdef \|(1+\|\cdot \|)^{\alpha} g \|_{1}=\int_{\R^d} (1+\|\V x\|)^{\alpha} g(\V x)\dint \V x
$$
where the switch to a negative exponent $-\alpha\le 0$ now demands that $g$ decays at infinity; e.g., a sufficient requirement for $g\in L_{1,-\alpha}(\R^d)$ is $|g(\V x)|\le \frac{C}{(1+\|\V x\|)^{\alpha+ d+ \epsilon}}$ with $\epsilon>0$.

Finally, since our primary interest is with ordinary functions that are well-defined pointwise, we also consider the Banach space
$$
C_{\rm{b},\alpha}(\R^d)=\left\{f: \R^d\to \R \mbox{ continuous and s.t. }  \|f\|_{\infty,\alpha}<+\infty\right\}
$$
that shares the same norm as $L_{\infty,\alpha}(\R^d)$, but whose elements (functions) are constrained to be continuous.
The hierarchy between these various spaces is described by the embedding relations
\begin{align*}
C_{\rm{b},\alpha}(\R^d) \embedIso L_{\infty,\alpha}(\R^d) \embedC \Spc S'(\R^d)\\
\Spc S(\R^d) \embedD L_{1,-\alpha}(\R^d) \embedIso \big(L_{\infty,\alpha}(\R^d)\big)',
\end{align*}
with additional explanations given in Appendix A.

%
\subsection{Characterization of linear operators and their adjoint}

A linear operator
is represented by an upright capital letter such as $\Lop$ or $\Op G$. Let $\Spc X$ and $\Spc Y$ be two topological vector spaces with continuous duals $\Spc X'$ and $\Spc Y'$, respectively. Then, the notation $\Op G: \Spc X \toC \Spc Y$ indicates that $\Op G$
continuously maps $\Spc X \to \Spc Y$. The adjoint of $\Op G$ is the unique operator $\Op G^\ast: \Spc Y' \toC \Spc X'$ such that
$\langle \Op G  x, y' \rangle_{\Spc Y \times \Spc Y'}=\langle x, \Op G^\ast y' \rangle_{\Spc X \times \Spc X'}$ for any $x\in \Spc X$ and $y'\in \Spc Y'$.
Moreover, if $\Spc X$ and $\Spc Y$ are both Banach spaces, then the norm of the operator is preserved \cite{Rudin1991}, 
as expressed by
$$
\| \Op G\|_{\Spc X\to \Spc Y}\eqdef \sup_{\varphi \in \Spc X\backslash\{0\}} \frac{\|\Op G \varphi \|_{\Spc Y} }{\|\varphi\|_\Spc X}=\| \Op G^\ast \|_{\Spc Y'\to \Spc X'}.
$$
The attractiveness of the Schwartz-Banach setting with $\Spc S(\R^d) \embedD \Spc X$ is the convenience of being able to specify the operator
in the more constrained---but mathematically foolproof---scenario $\Op G: \Spc S(\R^d) \toC \Spc Y \embedC \Spc S'(\R^d)$.
The critical property,  then, is the existence of a bound of the form
\begin{align}
\| \Op G \varphi \|_{\Spc Y}\le C \|\varphi\|_{\Spc X} \quad \mbox{for all} \quad \varphi \in \Spc S(\R^d),
\end{align}
which allows one to extend the domain of the operator to the complete space $\Spc X$ by continuity; i.e., $\Op G: \Spc X \toC \Spc Y$ with 
$\| \Op G\|_{\Spc X\to \Spc Y}\le C$.
%
This is a powerful extension principle that relies on the bounded-linear-transformation (B.L.T.) theorem.

\begin{theorem}[{\cite[Theorem I.7, p. 9]{Reed1980}}]
\label{Theo:BLT}
Let $\Op G$ be a bounded linear transformation from a normed space $(\Spc Z,\|\cdot\|_\Spc Z)$ to a complete normed space $(\Spc Y,\|\cdot\|_\Spc Y)$. Then, $\Op G$ has a unique extension to a bounded linear transformation (with the same bound) 
from the completion
of $(\Spc Z,\|\cdot\|_\Spc Z)$ to $(\Spc Y,\|\cdot\|_\Spc Y)$.
\end{theorem} 

The other foundational result that supports the present construction is Schwartz' kernel theorem \cite{Gelfand-Villenkin1964}, which states that the application of the operator $\Op G: \Spc S(\R^d) \toC \Spc S'(\R^d)$ to $\varphi$
yields a tempered distribution $\Op G\{\varphi\}: \Spc S(\R^d) \to \R$ 
specified by
  \begin{align}
\label{Eq:kernelrep}
\psi \mapsto \langle\Op G\{\varphi\},\psi \rangle=\langle g(\cdot,\cdot),\psi\otimes \varphi\rangle,
\end{align}
where $g(\cdot,\cdot)\in \Spc S'(\R^d \times \R^d)$ is the kernel of the operator
and $\psi\otimes \varphi\in \Spc S(\R^d \times \R^d)$ with $(\psi \otimes \varphi)(\V x,\V y)\eqdef \psi(\V x)\varphi(\V y)$. This property is often symbolized by the formal ``integral'' equation
\begin{align}
\label{Eq:kernelrep2}
\Op G: \varphi \mapsto \int_{\R^d} g(\cdot,\V y) \varphi(\V y)\dint \V y
\end{align}
with a slight abuse of notation. 
Conversely, 
the right-hand side of
\eqref{Eq:kernelrep2} defines a continuous operator $\Spc S(\R^d) \toC \Spc S'(\R^d)$ for any $g(\cdot,\cdot) \in \Spc S'(\R^d \times \R^d)$.
The association between the kernel and the operator is unique and, hence,
transferable to $\Op G: \Spc X \toC \Spc Y$ by continuity\footnote{The kernel is also called the generalized impulse response of the operator. Formally, this is indicated as $g(\cdot,\V y)=\Op G\{\delta(\cdot-\V y)\}$ with a slight abuse of notation when $\delta(\cdot-\V y) \notin \Spc X$. }.
Likewise, we may describe the adjoint $\Op G^\ast: \Spc Y' \toC \Spc X'\embedC \Spc S'(\R^d)$ by means of the relation
$$
\Op G^\ast: \varphi \mapsto \int_{\R^d} g(\V y,\cdot) \varphi(\V y)\dint \V y$$
where the flipping of the indices of the kernel is the infinite-dimensional counterpart of the transposition of a matrix.


The bottom line is that the consideration of a
Schwartz-Banach space allows for a concrete and unambiguous description of linear functionals and 
linear operators in terms of generalized functions and generalized kernels, respectively.

\begin{proposition}[Representation of functionals and operators]
\label{Prop:Concrete}
Let $\Spc X$ be a Schwartz-Banach space.
Then, any continuous linear functional
$g: \Spc X \to \R$ is uniquely characterized by a single element
$\left.g\right|_{\Spc S(\R^d)}\in \Spc S'(\R^d)$: the restriction of $g$
to $\Spc S(\R^d) \to \R$. 
Likewise, any continuous linear operator $\Op G: \Spc X \to \Spc Y$ where
$\Spc Y$ is a Banach subspace of $\Spc S'(\R^d)$
 is uniquely characterized by a single element $g(\cdot,\cdot)\in \Spc S'(\R^d \times \R^d)$,
which is the Schwartz kernel of the restriction $\left.\Op G\right|_{\Spc S(\R^d)}: \Spc S(\R^d) \to \Spc Y\embedC \Spc S'(\R^d)$. \end{proposition}
\begin{proof}
The Schwartz-Banach property implies that $\Spc S(\R^d) \embedD \Spc X$. Hence,
both $g$ and $\Op G$ are fully characterized by their restriction on $\Spc S(\R^d)$.
As such, $g$ is identified as an element of $\Spc S'(\R^d)$, while
$\Op G$ is uniquely characterized by its
 Schwartz kernel $g(\cdot,\cdot) \in \Spc S'(\R^d \times \R^d)$ when seen as an operator
 from $\Spc S(\R^d)$ to $\Spc S'(\R^d)$.
%
%
%
%
%
%

 \end{proof}
The common practice is to indicate the correspondences in Proposition \ref{Prop:Concrete} by \eqref{Eq:dualproduct} and \eqref{Eq:kernelrep2}, respectively.
%
One should keep in mind, however, that the rigorous interpretation of these relations involves the limit/extension process:
\begin{align*}
\langle g, f\rangle =\int_{\R^d} g(\V y) f(\V y)\dint \V y\eqdef  \lim_{i} \langle g, f_i\rangle \\
\int_{\R^d} g(\cdot,\V y) f(\V y)\dint \V y\eqdef   \lim_{i} \int_{\R^d} g(\cdot,\V y) f_i(\V y)\dint \V y
\end{align*}
for any sequence $(f_i)$ in $\Spc S(\R^d)$ such that
$\|f-f_i\|_{\Spc X}\to 0$ as $i\to \infty$, with the implicit understanding that the underlying ``integrals'' are the symbolic representations of 
continuous linear functionals acting on $f$ (resp. $f_i$).


\section{Spline-admissible operators}

The identification of the native space for an admissible operator $\Lop$ essentially boils down to the characterization of the solutions of the linear differential equation
$\Lop f=w$ for a suitable class of excitations $w\in \Spc X'$. 
This requires that $\Lop$ be invertible in an appropriate sense. As a minimum, we ask that $\Lop$ be injective (with a well-defined inverse $\Lop^{-1}$) when we restrict its domain to $\Spc S(\R^d)$.

\begin{definition} [Admissible operator]
\label{Def:splineadmis3}
A linear operator $\ \Op L: \Spc S(\R^d) \to \Spc S'(\R^d)$ 
is called {\em spline-admissible} 
if there exist an order $\alpha\inR$ of algebraic growth, an inverse operator $\Op \Lop^{-1}$, and a finite-dimensional space
$\Spc N_{\V p}={\rm span}\{p_n\}_{n=1}^{N_0}$ such that
\begin{enumerate}
\item \label{Item:Laststability} $\Lop$ and $\Lop^\ast$ continuously map $\Spc S(\R^d)\to L_{1,-\alpha}(\R^d)$. 
(The continuity of the adjoint is required for the extended version of the operator
$\Lop: L_{\infty,\alpha}(\R^d)\to \Spc S'(\R^d)$ in Condition \ref{Item:FiniteDimNullSpace} to be well defined.)

\item \label{Item:Inversestability} $\Lop^{-1}$ and $\Lop^{-1\ast}$ continuously map $L_{1,-\alpha}(\R^d)\to L_{\infty,\alpha}(\R^d)$. 
(The assumption $\Lop^{-1}: L_{1,-\alpha}(\R^d)\toC L_{\infty,\alpha}(\R^d)$ is actually sufficient as it implies
that $\Lop^{-1\ast}: L_{1,-\alpha}(\R^d)\toC L_{\infty,\alpha}(\R^d)$ by duality.)


\item \label{Item:Invertibility}  Invertibility: $\Lop^{-1\ast}\Lop^\ast \varphi=\varphi$ and 
$\Lop^{-1}\Lop \varphi=\varphi$ for all $\varphi \in \Spc S(\R^d)$.

\item \label{Item:FiniteDimNullSpace} Restricted null space: $\Spc N_{\V p}=\Spc N_{\Lop }= \{p \in L_{\infty,\alpha}(\R^d): \Lop\{p\}=0\}.$
\end{enumerate}
\end{definition}

A preferred scenario is when $\Lop$ and $\Lop^{-1}$ are both linear-shift-invariant---that is, convolution operators---with respective frequency response $\hat L(\bw)$ and $1/\hat L(\bw)$. Then,
$\varphi \mapsto \Lop^{-1}\varphi=\rho_\Lop\ast \varphi\eqdef \int_{\R^d} \rho_\Lop(\V y) \varphi(\cdot-\V y)\dint \V y$, where
$$
\rho_\Lop(\V x)=\Fourier^{-1}\left\{\frac{1}{\hat L(\bw)} \right\}(\V x),
$$
which is the (generalized) inverse Fourier transform of $1/\hat L\in \Spc S'(\R^d)$ under the implicit assumption that the latter is a well-defined tempered distribution. In that case, we refer to 
$\rho_\Lop$, which satisfies the formal property $\Lop\{\rho_\Lop\}=\delta$, as the canonical Green's function of $\Lop$.

\begin{example} The derivative operator $\Dop=\frac{\dint}{\dint x}$ with $\Dop^\ast=-\Dop$ is spline-admissible with $N_0=1$ and $\alpha=0$.
Indeed, $\Dop: \Spc S(\R)\toC\Spc S(\R)\embedC L_1(\R)$, while its null space is $\Spc N_\Dop={\rm span}\{p_1\}\subset C_{\rm b}(\R) \subseteq L_{\infty,0}(\R)=L_{\infty}(\R)$ with $p_1(x)=1$.
Its canonical inverse is given by
$\Dop^{-1}: \varphi \mapsto \rho_{\Dop} \ast \varphi$ with $ \rho_{\Dop}(x)=\frac{1}{2} {\rm sign}(x)$ and, hence, such that
$\Dop^{-1}: L_1(\R)\toC L_\infty(\R)$.
 \end{example}

Let us now briefly comment on the assumptions in Definition \ref{Def:splineadmis3}. Conditions \ref{Item:Laststability} and \ref{Item:Inversestability}
ensures that the composition of operators in Condition \ref{Item:Invertibility} (\ie $\Lop^{-1\ast}\Lop^\ast: \Spc S(\R^d) \toC L_{1,-\alpha}(\R^d) \toC \Spc S'(\R^d)$ and $
\Lop^{-1}\Lop: \Spc S(\R^d) \toC L_{1,-\alpha}(\R^d) \toC \Spc S'(\R^d)$) are legitimate. 
Condition \ref{Item:Laststability} limits the framework to operators that do not drastically affect decay. However, it does not penalize a loss of regularity. In particular, it is met by all constant-coefficient (partial) differential operators, which happen to be continuous $\Spc S(\R^d)\to \Spc S(\R^d) \embedC L_{1,-\alpha}(\R^d)$.
The condition with an adequate $\alpha$ is also satisfied by the fractional derivative operators $\Dop^\gamma$ or $(-\Delta)^{\gamma/2}$ (fractional Laplacian) whose impulse responses decay
like $1/\|\V x\|^{\gamma+d}$. 
This decay property can be used to show that the fractional derivatives $\Dop^\gamma$ with $\gamma>0$ are continuous $\Spc S(\R^d) \to L_{1,-(\gamma+d)}(\R^d)$ \cite{Unser2007}, while their Green's functions are included in $L_{\infty,\alpha}(\R^d)$ with $\alpha=(\gamma-d)$, which is a favourable state of affairs in the context of Condition \ref{Item:Laststability}.

An important observation is that the left-invertibility of $\Lop^\ast$ in Condition \ref{Item:Invertibility}  is equivalent
to $\Lop^{}\Lop^{-1}\varphi=\varphi$ for all $\varphi \in \Spc S(\R^d)$ with $\Lop^{-1}\varphi\in L_{\infty,\alpha}(\R^d)$ and $\Lop: L_{\infty,\alpha}(\R^d) \to \Spc S'(\R^d)$ extended accordingly by duality. Indeed, for any $\tilde \varphi \in \Spc S(\R^d)$,
we have that
\begin{align}
\label{Eq:InvertibleL}
\langle \varphi, \tilde\varphi \rangle=\langle\varphi,  \Lop^{-1\ast}\Lop^\ast \tilde \varphi \rangle=\langle \Lop\Lop^{-1} \varphi, \tilde\varphi \rangle,
\end{align}
where the factorization is legitimate, in reason of the compatible range and domain of the underlying operators. We then invoke the denseness of $\Spc S(\R^d)$ in $\Spc S'(\R^d)$ (see Proposition \ref{Prop:Sdense}) to extend the validity of \eqref{Eq:InvertibleL} for all $\tilde \varphi \in \Spc S'(\R^d)$, which yields the desired result.

While Condition \ref{Item:Invertibility} tells us that $\Lop$ is invertible on $\Spc S(\R^d)$, it does not guarantee that the property still holds when the domain is extended to $L_{\infty,\alpha}(\R^d)$
because the corresponding null space $\Spc N_\Lop=\Spc N_{\V p}$ in Condition \ref{Item:FiniteDimNullSpace} may be nontrivial. We shall resolve this ambiguity by 
factoring out the null-space components. Our approach is based on the construction 
of a projection operator $\Proj_{\Spc N_\V p}: L_{\infty,\alpha}(\R^d) \toC \Spc N_\V p$, as described next. The latter is then used in Section \ref{Sec:PseudoInverse2} to
specify a stable pseudo-inverse operator that is a corrected (or regularized) version of $\Lop^{-1}$.

\subsection{Biorthogonal system for the null space of $\Lop$}
\label{Sec:NulSpace}
%

The fundamental hypothesis here 
is that the growth-restricted null space $\Spc N_\Lop=\{f \in L_{\infty,\alpha}(\R^d): \Lop \{p\}= 0\}$ 
admits a basis $\V p=(p_1,\dots,p_{N_0})$ of finite dimension $N_0$.
The idea then is to select a 
set of analysis functionals  
$\phi_1,\dots,\phi_{N_0} \in \Spc Y 
\subseteq
 \Spc S'(\R^d)$ 
that satisfy the biorthogonality relation
\begin{align}
\label{Eq:biortho}
\langle \phi_m,p_n\rangle=\delta_{m,n},\quad m,n\in\{1,\dots,N_0\}.
\end{align}
To give a concrete meaning to the above duality product, we are assuming the existence\footnote{In view of the inclusion
$\Spc N_{\Lop}\subset L_{\infty,\alpha}(\R^d)$, the natural choice is to take $(\Spc Y, \Spc Y')=\big(L_{1,-\alpha}(\R^d),L_{\infty,\alpha}(\R^d)\big)$.
We shall see that this is extendable to $(\Spc Y, \Spc Y')=(\Spc X_\Lop,\Spc X'_\Lop)$.}
of a dual pair $(\Spc Y, \Spc Y')$ of subspaces of $\Spc S'(\R^d)$ such that
$\phi_1,\dots,\phi_{N_0} \in \Spc Y$ and $p_1,\dots,p_{N_0} \in \Spc Y'$.
Such a construction is always feasible (see Proposition \ref{Prop:biorthoexist}) with the 
choice of $\V \phi$ being free.
We also understand that there is a whole equivalence class of representations of $\Spc N_{\Lop}=\Spc N_{\V p}=\Spc N_{\tilde{\V p}}$ with 
$\tilde {\V p}=\M B{\V p}$, under the constraint that the matrix $\M B\in \R^{N_0\times N_0}$ be invertible.

\begin{definition}[Biorthogonal system] 
\label{Def:biorthogonalsystem}

A pair $(\V \phi,\V p)$ with $\V \phi=(\phi_1,\dots,\phi_{N_0})\in \Spc Y^{N_0}$ and $\V p=(p_1,\dots,p_{N_0})\in (\Spc Y')^{N_0}$  is called a {\em biorthogonal system} for the finite-dimensional subspace $\Spc N_{\V p}={\rm span}\{p_n\}_{n=1}^{N_0}\subset \Spc Y' \subseteq \Spc S'(\R^d)$ if any
$p \in \Spc N_{\V p}$ admits a unique expansion of the form
\begin{align}
\label{Eq:biorthoexpand}
p=\sum_{n=1}^{N_0} \langle \phi_n, p\rangle p_n.
\end{align}
The natural norm induced on $\Spc N_\V p$ by such a system is
$$\|p\|_{\Spc N_\V p}=\|\V \phi(p)\|_2=\left(\sum_{n=1}^{N_0}|\langle \phi_n, p\rangle|^2\right)^\frac{1}{2}.$$ 
The system is said to be {\em universal} if $\phi_n \in \Spc S(\R^d)$ for $n=1,\dots,N_0$.
\end{definition}
The unicity of the representation in Definition \ref{Def:biorthogonalsystem} implies that $\V p$ be a basis of $\Spc N_\V p$, while the validity of \eqref{Eq:biorthoexpand} for $p=p_n$ implies
that the underlying functions be biorthogonal, as expressed by \eqref{Eq:biortho}.

Our next proposition ensures the existence of such systems for any given basis $\V p$
under our working hypothesis $\Spc Y\supseteq \Spc S(\R^d)$.

\begin{proposition}[Existence of biorthogonal systems]
\label{Prop:biorthoexist}
Let $\Spc N_\V p={\rm span}\{p_n\}_{n=1}^{N_0}$ be a $N_0$-dimensional subspace of $\Spc S'(\R^d)$.
Then, there always exists some (universal) biorthogonal set of functionals $\phi_1,\dots,\phi_{N_0} \in \Spc S(\R^d)$.
\end{proposition}
\begin{proof}
This result is deduced from the following variant of the Hahn-Banach theorem
with $\Spc V=\Spc S'(\R^d)$.
\begin{theorem}[{\cite[Theorem 3.5] 
{Rudin1991}}]
\label{Theo:biorthoexist0}
Let $\Spc Z$ be a linear subspace of a topological vector space $\Spc V$, and $v_0$ be an element of $\Spc V$. If $v_0$ is not in the closure of $\Spc Z$, then there exists a continuous linear functional $\phi$ on $\Spc V$ such that $\langle \phi, v_0\rangle=1$ but $\langle \phi, z\rangle=0$ for every $z \in \Spc Z$.
\end{theorem}
We then proceed 
by successive exclusion of $v_0=p_n$ with $\phi=\phi_n$ and $\Spc Z={\rm span}_{m\neq n}\{p_m\}$, while the finite dimensionality of $\Spc Z$ and the linear independence of the $p_m$ ensures that $p_n \notin \overline{\Spc Z}=\Spc Z$.
\end{proof}
 \begin{table}
 \begin{tabular}{p{6cm}p{2.75cm}p{2.5cm}}
 \hline\\[-2ex]
    Description & Operator & Kernel\\[1ex]
    \hline\\[-2ex]   
Riesz map $\ \Spc N_{\V \phi} \to \Spc N_{\V p}=\Spc N'_{\V \phi}$ & $\Op R_{\V p}$ & $\displaystyle \sum_{n=1}^{N_0} p_n(\V  x)  p_n(\V y)$ \\[2ex]
Riesz map $\ \Spc N_{\V p} \to \Spc N_{\V \phi}=\Spc N'_{\V p}$ & $\Op R_{\V \phi}$ & $\displaystyle \sum_{n=1}^{N_0} \phi_n(\V  x)  \phi_n(\V y)$ \\[2ex]
Projector $\ \Spc X_\Lop'\subseteq L_{\infty,\alpha}(\R^d) \to \Spc N_{\V p}$ & ${\rm Proj}_{\Spc N_{\V p}}=\Op R_{\V p}\Op R_{\V \phi}$ & $\displaystyle \sum_{n=1}^{N_0} p_n(\V  x)  \phi_n(\V y)$ \\[2ex]
Projector $\ \Spc X_\Lop\supseteq L_{1,-\alpha}(\R^d) \to \Spc N_{\V \phi}$ & ${\rm Proj}_{\Spc N_{\V \phi}}=\Op R_{\V \phi}\Op R_{\V p}$ & $\displaystyle \sum_{n=1}^{N_0} \phi_n(\V  x)  p_n(\V y)$ \\[2.5ex]
    \hline
\end{tabular}
\caption{\label{tab:ProjOperators} Complete set of operators associated with the biorthogonal system $(\V \phi,\V p)$
with $\Spc N_{\V p}={\rm span}\{p_n\}_{n=1}^{N_0}$ and $\Spc N_{\V \phi}={\rm span}\{\phi_n\}_{n=1}^{N_0}$.}
\end{table}

\begin{example} We recall that the derivative operator $\Dop=\frac{\dint}{\dint x}$ from Example 1 is spline-admissible with $\alpha=0$, $N_0=1$, 
 and $\Spc N_\Dop={\rm span}\{p_1\} \subset L_{\infty}(\R)$ where $p_1(x)=1$.
As complementary analysis functional, we may choose
any function $\phi_1 \in L_{1}(\R)$ (or $\phi_1 \in \Spc S (\R)$ if we are aiming at a universal solution)
such that $\int_{\R} \phi_1(x)\dint x=\langle \phi_1, 1\rangle=1$. 
\end{example}

\begin{example}
The $N_0$th-order derivative operator $\Lop=\Dop^{N_0}: \Spc S(\R)\toC\Spc S(\R)$ is spline-admissible with
 $\Lop^{-1}: \varphi \mapsto \rho_{\Dop^{N_0}} \ast \varphi$ where $\rho_{\Dop^{N_0}}(x) =\frac{1}{2} {\rm sign}(x)\frac{x^{N_0-1}}{(N_0-1)!}$, $\alpha=(N_0-1)$, and $\Spc N_{\Dop^{N_0}}={\rm span}\{p_n\}_{n=1}^{N_0}\subset L_{\infty,N_0-1}(\R)$ with
 $p_n(x)=x^{n-1}$. A possible choice of universal biorthogonal system is 
$(\tilde{\V \phi},\tilde{\V p})$  with 
$\tilde \phi_n(x)=\tilde p_n(x) \ee^{-x^2/2} \in \Spc S(\R)$ and $\tilde p_1,\dots,\tilde p_{N_0}$ the normalized Hermite polynomials of degree $0$ to $(N_0-1)$. By definition, the latter form
a basis of $\Spc N_{\Dop^{N_0}}$---the space of polynomials of degree $(N_0-1)$---that fullfils the 
(bi)-orthogonality relation $\langle \ee^{-x^2/2} \tilde p_n, \tilde p_m\rangle=\delta[m-n]$.
\end{example}

Given such a system, we then specify
$\Spc N_{\V p}$ 
and 
$\Spc N_{\V \phi}={\rm span}\{\phi_n\}_{n=1}^{N_0}$
as a dual pair of $N_0$-dimensional Hilbert spaces equipped with the inner products $\langle p,q\rangle_{\Spc N_{\V p}}=\langle \Op R_{\V \phi}\{p\}, q\rangle$ and
$\langle \phi,\varphi\rangle_{\Spc N_{\V \phi}}=\langle \Op R_{\V p}\{\phi\}, \varphi\rangle$, respectively, where
\begin{align*}
	\Op R_{\V p}\{\phi\}&=\sum_{n=1}^{N_0} \langle p_n, \phi\rangle p_n\\
\Op R_{\V \phi}\{p\}&=\sum_{n=1}^{N_0} \langle \phi_n, p \rangle \phi_n.
\end{align*}
Indeed, $\phi_n=\Op R_{\V \phi}\{p_n\}$ and $p_n=\Op R_{\V p}\{\phi_n\}$, which allows us to identify
$\Op R_{\V \phi}$ (resp. $\Op R_{\V p}$) as the Riesz map $\Spc N_{\V p}\to\Spc N_{\V \phi}=\Spc N'_{\V p}$
(resp. $\Op R_{\V p}: \Spc N_{\V \phi}\to\Spc N_{\V p}=\Spc N'_{\V \phi}$). 
For the proper interpretation of this (Riesz) pairing, we recall that the dual space $\Spc N'_{\V p}$ is an abstract entity that is formed by the linear functionals that are continuous on $\Spc N_{\V p}$; strictly speaking, each functional is an equivalence class of tempered distributions. The notation $\Spc N'_{\V p}=\Spc N_{\V \phi}$ indicates 
that $\Spc N'_{\V p}$ is isometrically isomorphic to $\Spc N_{\V \phi}$, meaning that each (abstract) member of $\Spc N'_{\V p}$ has a unique representative in 
$\Spc N_{\V \phi}$, which then serves as our concrete descriptor.
The generic elements of these spaces are denoted by
$p \in \Spc N_{\V p}$ with $\|p\|_{\Spc N_{\V p}}=\|\V \phi(p)\|_2$, where
$\V \phi(p)=(\langle \phi_1,p\rangle,\dots, \langle \phi_{N_0},p\rangle)\inR^{N_0}$ and $\phi \in \Spc N_{\V \phi}$ with $\|\phi\|_{\Spc N_{\V \phi}}=\|\V p(\phi)\|_2$.

Under the working assumptions that $\Spc N_{\V \phi} \subseteq \Spc Y$ and
$\Spc N_\V p \subseteq 
\Spc Y'$, where $\Spc Y$ is a suitable Schwartz-Banach space,
the domain of continuity of these operators is extendable to 
$\Op R_{\V p}: \Spc Y \toC \Spc N_\V p$ and $\Op R_{\V \phi}: \Spc Y' \toC \Spc N_\V \phi$.
The corresponding projection operators are
\begin{align}
\Proj_{\Spc N_\V \phi}\{g\}&=\sum_{n=1}^{N_0} \langle p_n, g\rangle \phi_n=\Op R_{\V \phi}\Op R_{\V p}\{g\} \label{Eq:ProjNphi},\\
\Proj_{\Spc N_\V p}\{f\}&=\sum_{n=1}^{N_0} \langle \phi_n, f\rangle p_n=\Op R_{\V p}\Op R_{\V \phi}\{f\}\
\end{align}
for any $g \in \Spc Y$ and $f \in \Spc Y'$
with the property that $\Proj_{\Spc N_\V \phi}\{\phi\}=\phi$ for all $\phi \in \Spc N_\V \phi$ and $\Proj_{\Spc N_\V p}\{p\}=p$ for all $p \in \Spc N_\V p$.
%
We also note that $\Proj^\ast_{\Spc N_\V \phi}=\Proj_{\Spc N_\V p}$, which emphasizes the symmetry of the construction.

We can also rely on the generic duality bound $|\langle p_n,g\rangle|\le \|p_n\|_{\Spc Y'} \|g\|_{\Spc Y}$ to get a handle on the continuity properties of these operators.
Specifically, based on \eqref{Eq:ProjNphi}, we find that
\begin{align}
\left\|\V p\big(\Proj_{\Spc N_\V \phi}\{g\}\big)\right\|_2&=
\left(\sum_{n=1}^{N_0} |\langle p_n,g\rangle|^2 \right)^{1/2} \nonumber\\
&\le \left( \sum_{n=1}^{N_0} \|p_n\|^2_{\Spc Y'}\right)^{1/2} \|g\|_{\Spc Y} \label{Eq:ProjectorBound}
\end{align}
where the leading constant on the right-hand side provides an upper bound on the norm of the operator $\Proj_{\Spc N_\V \phi}: \Spc Y \toC \Spc N_\V \phi\subset \Spc Y$.
Likewise, we have that $\Proj_{\Spc N_\V p}: \Spc Y' \toC \Spc N_\V p\subseteq \Spc Y'$
with $\|\Proj_{\Spc N_\V p}\|_{\Spc Y'\to\Spc N_\V p} \le \left( \sum_{n=1}^{N_0} \|\phi_n\|^2_{\Spc Y}\right)^{1/2}$.

The whole setup is summarized in Table \ref{tab:ProjOperators} for the choice of spaces $\Spc Y=L_{1,-\alpha}(\R^d)$ and $\Spc Y'=L_{\infty,\alpha}(\R^d)$, which is adopted for the sequel. 
\subsection{Specification of a suitable pseudo-inverse operator}
\label{Sec:PseudoInverse2}

Under our hypotheses,
the projection 
operator $\Proj_{\Spc N_\V p}$ defined by \eqref{Eq:ProjNphi}
is continuous $\Spc Y'=L_{\infty,\alpha}(\R^d)\toC L_{\infty,\alpha}(\R^d)$. This, in turn, ensures the continuity of
$$\Lop_{\V \phi}^{-1}\eqdef(\Identity-\Proj_{\Spc N_\V p}) \circ \Lop^{-1}: L_{1,-\alpha}(\R^d) \toC L_{\infty,\alpha}(\R^d)\toC L_{\infty,\alpha}(\R^d).$$ 
By observing that $(\Identity-\Proj_{\Spc N_\V p})^\ast=(\Identity-\Proj_{\Spc N_\V \phi})$,
we then readily deduce that the adjoint of $\Lop_{\V \phi}^{-1}$ is such that
\begin{align}
\label{Eq:PseudoInverseDef}
\Lop_{\V \phi}^{-1\ast}=\Lop^{-1\ast}(\Identity-\Proj_{\Spc N_\V \phi}) : L_{1,-\alpha}(\R^d)  \toC L_{\infty,\alpha}(\R^d),
\end{align}
owing to the property that $L_{1,-\alpha}(\R^d) \embedIso \big(L_{1,-\alpha}(\R^d)\big)''=\big(L_{\infty,\alpha}(\R^d)\big)'$.

While \eqref{Eq:PseudoInverseDef}  ensures that $\Lop_{\V \phi}^{-1\ast}$ is well-defined on $L_{1,-\alpha}(\R^d)$, we also want to make sure
that the range of this operator can be restricted to the primary space $\Spc X$ (the pre-dual of $\Spc X'$ mentioned in the introduction), which calls for some additional compatibility hypotheses.

\begin{definition} [Compatibility of $(\Lop,\Spc Y)$]
\label{Def:Hypotheses}
Let $\Lop$ be a linear operator and $\Spc Y$ a Banach subspace of $\Spc S'(\R^d)$.
We say that the pair $(\Lop, \Spc Y)$ is compatible if
\begin{enumerate}
\item There exists a Schwartz-Banach space $\Spc X$ (see Definition \ref{Def:SchwartzBanach}) such that $\Spc Y=\Spc X'$;
\item the operator $\Lop$ is spline-admissible with order $\alpha$ (see Definition \ref{Def:splineadmis3});
\item the adjoint operator $\Lop^\ast$ is continuous $\Spc X\toC \Spc S'(\R^d)$ and injective;
\item there exists a universal 
biorthogonal system $(\V \phi,\V p)$ 
of the null space of $\Lop$ 
such that
$
\Lop^{-1\ast}_\V \phi 
: L_{1,-\alpha}(\R^d) \toC \Spc X
$
where 
$\Lop^{-1\ast}_\V \phi$ is specified by \eqref{Eq:PseudoInverseDef}. This condition is called $\Spc X$-stability. 
\end{enumerate}
%
%



\end{definition}
Conditions 1 and 2 are explicit and hardly constraining. The injectivity of $\Lop^\ast$ 
is equivalent to the intersection between the extended null space of $\Lop^\ast$ and $\Spc X$ being trivial. 
The most constraining requirement is Condition 4, which needs to be checked on a case-by-case basis. Interestingly, we shall see that, if the condition holds for one particular biorthogonal system
$(\V \phi,\V p)$ with $\Spc N_\V \phi \subset \Spc S(\R^d)$ (universality property), then it also holds for any other admissible biorthogonal system
$(\tilde{\V \phi},\tilde{\V p})$ as long as $\Spc N_{\tilde{\V \phi}}\subset\Spc X_\Lop$, where the pre-dual space $\Spc X_\Lop$ is characterized in Theorem \ref{Theo:pre-dual2}.


The operator $\Lop^{-1\ast}_\V \phi$ in \eqref{Eq:PseudoInverseDef} is fundamental to our construction.
The key will be to extend its domain to make it surjective over $\Spc X$. To that end, we now introduce
a suitable pre-Banach space that will then be completed to yield our pre-dual space $\Spc X_\Lop$.

\begin{definition}[Pre-Banach space for $(\Lop,\Spc X')$]
\label{Def:PreNative}
Under the compatibity hypotheses of Definition \ref{Def:Hypotheses},
we specify the pre-Banach space for $(\Lop,\Spc X')$ as the vector space
\begin{align}
\label{Eq:pre-dual1}
\Spc P_{\Lop}=\{\Lop^\ast \varphi + \phi: \varphi \in \Spc S(\R^d), \phi\in\Spc N_{\V \phi}\}\subseteq L_{1,-\alpha}(\R^d).
\end{align}

\end{definition}

\begin{proposition} 
\label{Prop:PreBanach}
Under the compatibity hypotheses of Definition \ref{Def:Hypotheses},
for any $g \in \Spc P_{\Lop}$, there is a unique pair $(\varphi=\Lop_\V \phi^{-1\ast}g,
\phi=\Proj_{\Spc N_\V \phi}\{g\})\in \Spc S(\R^d) \times \Spc N_\V \phi$ such that 
$g=\Op L^\ast \varphi + \phi$, where the two underlying operators are defined by \eqref{Eq:PseudoInverseDef} and \eqref{Eq:ProjNphi}, respectively. 
In particular, this implies the following:
\begin{enumerate}
\item Null-space property: $\Lop_\V \phi^{-1\ast}\phi=0$ for all $\phi \in \Spc N_\V \phi\subset \Spc P_\Lop$.
\item Left invertibility of $\Lop^\ast$: $\Lop_\V \phi^{-1\ast}\Lop^\ast\varphi=\Lop^{-1\ast}\Lop^\ast\varphi=\varphi$ for all $\varphi\in \Spc S(\R^d)$. 
\item Pseudo-right-invertibility: $\Lop^{\ast}\Lop_\V \phi^{-1\ast}g=(\Identity-\Proj_{\Spc N_\V \phi})\{g\}$ for all $g\in \Spc P_\Lop$.
\item $\Spc P_{\Lop}$ has a direct sum decomposition as
$\Spc P_{\Lop}=\Lop^\ast\big(\Spc S(\R^d)\big)\oplus \Spc N_\V \phi$.
\item $\Spc P_{\Lop}$, equipped
with the composite norm
\begin{align}
\label{Eq:XLnorm}
\|g\|_{\Spc X_\Lop}\eqdef \max(\underbrace{\|\Lop^{-1\ast}_\V \phi g\|_{\Spc X}}_{\|\varphi\|_\Spc X},\underbrace{\|\V p(\Proj_{\Spc N_\V \phi}\{g\})\|_{2}}_{\|\V p(\phi)\|_2}),
\end{align}
is a normed subspace of $L_{1,-\alpha}(\R^d)$. 
\item $\Lop^{-1\ast}_\V \phi$ 
is bounded $\Spc P_\Lop \to \Spc X$ with $\|\Lop^{-1\ast}_\V \phi g\|_{\Spc X}\le \| g\|_{\Spc X_\Lop}$ for all $g\in \Spc P_\Lop$.
\item $\Proj_{\Spc N_\V \phi}$ is bounded $\Spc P_\Lop\to\Spc N_\V \phi$ with $\|\Proj_{\Spc N_\V \phi}g\|_{\Spc X_\Lop}\le \| g\|_{\Spc X_\Lop}$
for all $g\in \Spc P_\Lop$.
\end{enumerate}
\end{proposition}

Before moving to the proof, we observe that
every one of the operators $\Proj_{\Spc N_\V \phi}: L_{1,-\alpha}(\R^d) \toC \Spc N_\V \phi$, $\Lop_\V \phi^{-1\ast}: L_{1,-\alpha}(\R^d) \toC \Spc X$,
and $\Lop^\ast \circ \Lop_\V \phi^{-1\ast}: L_{1,-\alpha}(\R^d) \toC \Spc X\toC \Spc S'(\R^d)$ is well-defined on $g\in \Spc P_\Lop$ because $\Spc P_\Lop\subseteq L_{1,-\alpha}(\R^d)$
by construction. 

\begin{proof} From \eqref{Eq:pre-dual1}, $\Spc P_\Lop=\Lop^\ast\big(\Spc S(\R^d)\big)+\Spc N_\V \phi$.
Thanks to the hypotheses $p_n\in \Spc N_\Lop\subseteq L_{\infty,\alpha}(\R^d)=\big(L_{1,-\alpha}(\R^d)\big)'$ and $\psi=\Lop^\ast \varphi \in L_{1,-\alpha}(\R^d)$, we have that 
$\langle p_n,\psi \rangle=\langle p_n,\Lop^\ast \varphi \rangle=\langle \Lop p_n,\varphi \rangle=0$ or, equivalently, $\Proj_{\Spc N_\V \phi}\{\psi\}=0$ for all $\psi\in \Lop^\ast\big(\Spc S(\R^d)\big)$.
Likewise, due to the biorthogonality of $(\V \phi, \V p)$, $\Proj_{\Spc N_\V \phi}\{\phi\}=\phi$ for any $\phi \in \Spc N_\V \phi$.
In other words, for $g \in \Spc P_\Lop$, the condition $\Proj_{\Spc N_\V \phi}\{g\}=g$ is equivalent to $g \in \Spc N_\V \phi$; that is,
$\Lop^\ast\big(\Spc S(\R^d)\big)\cap\Spc N_\V \phi=\{0\}$, which establishes the direct-sum property (Item 4).

By applying the definition of $\Lop_\V \phi^{-1\ast}$ in \eqref{Eq:PseudoInverseDef}
and by invoking the spline-admissibility of $\Lop$, we then get that
$$\Lop_\V \phi^{-1\ast}\{\Op L^\ast \varphi + \phi\}=\Lop^{-1\ast}(\Identity-\Proj_{\Spc N_\V \phi})\{\psi + \phi\}=\Lop^{-1\ast}\psi=\Lop^{-1\ast}\Lop^\ast \varphi=\varphi.$$ 
for all $\varphi \in \Spc S(\R^d)$ and $\phi \in \Spc N_\V \phi$.
In doing so, we have actually shown that the 
map $\Spc S(\R^d) \times \Spc N_\V \phi  \to \Spc P_\Lop: (\varphi,\phi)  \mapsto g=\Op L^\ast \varphi + \phi$ is invertible
and that 
\begin{align}
\label{Eq:LpeudoInveseprop}
\Lop_\V \phi^{-1\ast}\psi=\Lop^{-1\ast}\psi \mbox{ for all }\psi=\Lop^\ast\varphi \in \Lop^\ast\big(\Spc S(\R^d)\big).
\end{align}
The properties in Items 1-3 then simply follow from the observation that
$(\Identity-\Proj_{\Spc N_\V \phi})\{\phi\}=0$ for any $\phi \in \Spc N_\V \phi$ which, in light of the previous identities, also yields
$\Lop_\V \phi^{-1\ast}(\Lop^\ast\varphi)=\Lop^{-1\ast}(\Lop^\ast\varphi)=\varphi$ 
for all $\varphi \in \Spc S(\R^d)$.

The property that $\Spc S(\R^d)\embedD \Spc X$ ensures that $\|\cdot\|_{\Spc X}$ is a {\em bona fide} norm over $\Spc S(\R^d)$. This allows us to equip $\Spc S(\R^d) \times \Spc N_\V \phi $ with the composite norm $(\varphi,\phi) \mapsto \max( \|\varphi\|_{\Spc X}, \|\V p(\phi)\|_2)$.
We then exploit the bijection between $\Spc S(\R^d)\times\Spc N_\V \phi$ and $\Lop^\ast\big(\Spc S(\R^d)\big)\oplus\Spc N_\V \phi$ to get
the expression of the norm given in \eqref{Eq:XLnorm}. 
In effect, this shows that the normed space $(\Spc S(\R^d),\|\cdot\|_{\Spc X}) \times \Spc N_\V \phi$
is isometrically isomorphic to $\Spc P_\Lop=\Lop^\ast\big(\Spc S(\R^d)\big) \oplus \Spc N_\Lop$ equipped with the $\|\cdot\|_{\Spc X_\Lop}$-norm, which is the desired result. 

To determine the corresponding norm inequalities (Items 5 and 6), we apply the direct-sum property to rewrite the norm of $g\in \Spc P_\Lop$ as
$$
\|g\|_{\Spc X_\Lop}=\max(\underbrace{\|(\Identity-\Proj_{\Spc N_\V \phi})\{g\}\|_{\Spc X_\Lop}}_{\|\Lop_\V \phi^{-1\ast}g\|_{\Spc X}}, \|\Proj_{\Spc N_\V \phi}\{g\}\|_{\Spc X_\Lop}),
$$
from which we deduce that $\|\Lop_\V \phi^{-1\ast}g\|_{\Spc X}\le \|g\|_{\Spc X_\Lop}$ and $\|\Proj_{\Spc N_\V \phi}\{g\}\|_{\Spc X_\Lop}\le \|g\|_{\Spc X_\Lop}$
for any $g \in \Spc P_\Lop$. These two  inequalities are sharp because
$\Proj_{\Spc N_\V \phi}$ (resp. $\Lop_\V \phi^{-1\ast}$) is an isometry over $\Spc N_\V \phi$ (resp.  
an isometry over $\Lop^\ast\big(\Spc S(\R^d)\big)$).
\end{proof}

We note that the projection property in Item 3 is equivalent to
$$\Lop^{\ast}\Lop_\V \phi^{-1\ast}\Lop^\ast\varphi=\Lop^\ast\varphi \quad \mbox{ for all }\quad \varphi \in \Spc S(\R^d),$$
which indicates that $\Lop_\V \phi^{-1\ast}$ is a generalized inverse of $\Lop^\ast$.

\section{Direct-sum topology of the native space}
\label{Sec:NativeSpace}
In this section, we reveal the Banach structure of the native space $\Spc X'_\Lop$ and 
of its pre-dual $\Spc X_\Lop$. We like to think of $\Spc X'_\Lop$ as the largest set of functions $f$ such that
$\|\Lop f\|_{\Spc X'}<\infty$ under the constraint of a finite-dimensional null space 
$\Spc N_\Lop=\{f \in \Spc X'_\Lop: \Lop\{f\}=0\}=\Spc N_{\V p}$. 
We also assume that $\Spc X$ is a Schwartz-Banach space with the cases of interest being
\begin{itemize}
\item $L_2(\R^d)$ equipped with the $\|\cdot\|_{L_2}$ norm and, more generally,
\item $L_p(\R^d)$ equipped with the $\|\cdot\|_{L_p}$ norm for $p\in[1,\infty)$;
\item $C_0(\R^d)$
, which can be specified as the closure of $\Spc S(\R^d)$ in the $\|\cdot\|_{\infty}$-norm.
\end{itemize}

\subsection{Pre-dual of the native space}
\begin{definition}
The pre-native space $\Spc X_\Lop$ is the completion of the pre-Banach space $\Spc P_\Lop$ of Definition \ref{Def:PreNative} for the ${\|\cdot\|_{\Spc X_\Lop}}$-norm defined by \eqref {Eq:XLnorm}.
\end{definition}
We now identify this space and prove that it is isometrically isomorphic to
$\Spc X \times \Spc N_\V \phi$,
which requires the use of Cauchy sequences to extend the properties of $\Spc P_\Lop$ in Proposition \ref{Prop:PreBanach}. We recall that the two normed spaces that underly the specification of $\Spc P_\Lop=\Spc W \oplus \Spc N_\V \phi$ in Proposition \ref{Prop:PreBanach}
are $\Spc W=(\Lop^\ast\big(\Spc S(\R^d)\big), \|\cdot\|_{\Spc W})$ with
$$
\|\psi\|_{\Spc W}= \|\Lop_\V\phi ^{-1\ast}\psi\|_{\Spc X},
$$
where $\Lop_\V\phi ^{-1\ast}: \Spc P_\Lop=(\Spc W \oplus \Spc N_\V \phi) \to \Spc X$ and
$(\Spc N_\V \phi,\|\cdot\|_{\Spc N_\V \phi})$ with
$\|\phi\|_{\Spc N_\V \phi}= \|\V p(\phi)\|_2$.


\begin{theorem}[Pre-dual of the native space]
\label{Theo:pre-dual2}
Under the admissibility and compatibility hypotheses of Definition \ref{Def:Hypotheses}, 
the completion of $(\Spc P_\Lop,{\|\cdot\|_{\Spc X_\Lop}})$ is the Banach space 
\begin{align*}
\Spc X_\Lop&=\Lop^\ast(\Spc X) \oplus \Spc N_\V \phi,
\end{align*}
which is itself isometrically
isomorphic to $\Spc X \times \Spc N_\V \phi$. Correspondingly,
the bounded operators $\Proj_{\Spc N_\V \phi}:\Spc P_\Lop \to \Spc N_\V \phi$
and $\Lop_\V \phi^{-1\ast} 
: \Spc P_\Lop \to \Spc X$ 
%
of Proposition \ref{Prop:PreBanach} have unique continuous extensions
$\Proj_{\Spc N_\V \phi}:\Spc X_\Lop \toC \Spc N_\V \phi$
and $\Lop_\V \phi^{-1\ast}: \Spc X_\Lop \toC \Spc X$
with the following properties:
\begin{enumerate}
\item Null space of $\Proj_{\Spc N_\V \phi}$: $\Spc U\eqdef \Lop^\ast(\Spc X)=\{g \in \Spc X_\Lop: \Proj_{\Spc N_\V \phi}\{g\}=0\}$.
\item Null space of $\Lop_\V \phi^{-1\ast}$: $\Spc N_\V \phi\eqdef {\rm span}\{\phi_n\}_{n=1}^{N_0}=\{g \in \Spc X_\Lop: \Lop_\V \phi^{-1\ast}g=0\}$.
\item Left inverse of $\Lop^\ast$: $\Lop_\V \phi^{-1\ast}\Lop^\ast v=v$  for any $v \in \Spc X$.
\item Pseudo-right-inverse: $\Lop^\ast \Lop_\V \phi^{-1\ast}g=(\Identity- \Proj_{\Spc N_\V \phi})\{g\}$ for any $g \in \Spc X_\Lop,$
\end{enumerate}
where $\Proj_{\Spc N_\V \phi}: g \mapsto \sum_{n=1}^{N_0} \langle p_n,g\rangle \phi_n$
with the property that $\langle p_n,u\rangle=0$ for all $u \in \Spc U$ and $\langle p_m,\phi_n\rangle=\delta[m-n]$.
Moreover, we have the hierarchy of continuous and dense embeddings $$\Spc S(\R^d) \embedD L_{1,-\alpha}(\R^d) \embedD \Spc X_\Lop \embedD \Spc S'(\R^d).$$

\end{theorem}
\begin{proof} 
By definition, we have that $\Spc X_\Lop=\overline{\Spc P}_\Lop$
which, in view of Theorem \ref{Theo:DirectSum} in Appendix B, is itself decomposable as $\overline{\Spc P}_\Lop=\overline{\Spc W} \oplus 
\overline{\Spc N}_\V \phi=\overline{\Spc W} \oplus \Spc N_\V \phi$ (because $\Spc N_\V \phi$ is finite-dimensional).

\item {\em (i) $\overline{\Spc W}=\Spc U=\Lop^\ast(\Spc X)=\{u=\Lop^\ast v: v \in \Spc X\}$ equipped with the topology inherited from $\Spc X$.}
\\
Since the map $\Lop^\ast: \Spc X \toC \Spc S'(\R^d)$ is injective, we have an isometric isomorphism
between the Banach space $(\Spc X,\|\cdot\|_\Spc X)$ and
$\Spc U$, which is itself a Banach space equipped with the norm $\|u\|_{\Spc U}=\|v\|_\Spc X$ where
$v$ is the unique element in $\Spc X$ such that $u=\Lop^\ast v$. In particular, for any $\psi \in \Spc W \subseteq \Spc U$, we can use 
Property 2 of Proposition \ref{Prop:PreBanach} (invertibility) 
to show that 
$$
\|\psi\|_{\Spc U}=\|\Lop^{-1\ast}_\V \phi\psi\|_{\Spc X}=\|\psi\|_{\Spc W}.
$$
Moreover, because the spaces
$(\Spc S(\R^d),\|\cdot\|_{\Spc X})$ and $(\Spc W,\|\cdot\|_{\Spc U})$ are isometric, 
$\Spc W$ is dense in $\Spc U$: for any $u=\Lop^\ast v \in \Spc U$ and $\epsilon>0$, there exists some
$\psi_\epsilon \in \Spc W$ such that $\|u-\psi_\epsilon\|_{\Spc U}\le \epsilon$. Indeed, the denseness of $\Spc S(\R^d)$ in $\Spc X$ implies the existence of
$\varphi_\epsilon\in \Spc S(\R^d)$ such that $\|v-\varphi_\epsilon\|_{\Spc X}=\| \Lop^\ast v-\Lop^\ast \varphi_\epsilon\|_{\Spc U}\le \epsilon$
so that it suffices to take $\psi_\epsilon=\Lop^\ast \varphi_\epsilon \in \Spc W$.
Since $\Spc U$ is complete and admits $\Spc W$ as a dense subset, it can be identified
as the completion of $\Spc W$ equipped with the $\|\cdot\|_{\Spc W}=\|\cdot\|_{\Spc U}$-norm.

\item {\em (ii) Extension of operators and functionals}\\
For clarity, 
we mark the extended operators mentioned in the theorem with a tilde.
Specifically, the application of Theorem \ref{Theo:BLT} with $\Spc Z=\Spc P_\Lop$ and $\Spc Y=\Spc X, \Spc N_\V \phi$, and $\R$ allows us to specify
the unique extensions
\begin{itemize}
\item $
\widetilde{\Lop_\V\phi ^{-1\ast}}: \overline{\Spc P}_\Lop \toC \Spc X$ with $\|\widetilde{\Lop_\V\phi ^{-1\ast}}\|= 1$
\item $
\widetilde{\Proj}_{\Spc N_\V \phi}: \overline{\Spc P}_\Lop \toC \Spc N_\V \phi$ with $\|\widetilde{\Proj}_{\Spc N_\V \phi}\|= 1$
\item $\widetilde{p}_n: \overline{\Spc P}_\Lop \to \R$ with $\|\widetilde{p}_n\|\le 1$.

\end{itemize}
The relevant bounds for the two first instances are directly deducible from Properties 6 and 7 in Proposition \ref{Prop:PreBanach}, while
the explicit definitions of these extensions are given in \eqref{E:PseudoInvExt} and \eqref{Eq:ExtendProj2}.
 As for the functionals $p_n: \Spc P_\Lop \to \R$ 
for $n=1,\dots,N_0$, 
we observe that
$$
|\langle p_n,\Lop^\ast\varphi + \phi\rangle|=|\langle p_n,\phi\rangle| 
\le \max( \|\varphi\|_{\Spc X} , \left(\sum_{m=1}^{N_0} |\langle p_m,\phi\rangle|^2\right)^\frac{1}{2}),
$$
for any $\varphi\in \Spc S(\R^d)$ and
$\phi
\in \Spc N_{\V \phi}$, which yields the supporting bound
$$|\langle p_n,g\rangle|\le \|g\|_{\Spc X_\Lop}\quad\mbox{for all}\quad g \in \Spc P_\Lop.$$



\item {\em (iii) Derivation of Properties 1-4 by continuity}\\
In line with the argumentation in Item ({\em i}), for any $u \in \Spc U=\overline{\Spc W}$, we have that
\begin{align}
\label{E:PseudoInvExt}\|u\|_{\Spc U}&=\|\widetilde{\Lop_{\V \phi}^{-1\ast}}u\|_{\Spc X} \quad  \mbox{ with } \quad 
\widetilde{\Lop_{\V \phi}^{-1\ast}}u\eqdef\lim_{i\to\infty} \Lop_{\V \phi}^{-1\ast}\psi_i,
\end{align}
where $(\psi_i)$ is any Cauchy sequence in $\Spc W$ such that $u=\lim_{i \to\infty} \psi_i \in \overline{\Spc W}=\Spc U$.
In fact, the underlying isometric isomorphism ensures that
a Cauchy sequence $(\psi_i)$ in $\Spc W$ maps to a corresponding sequence
 $(\varphi_i=\Lop_{\V \phi}^{-1\ast}\psi_i)$ that is Cauchy in $(\Spc S(\R^d),\|\cdot\|_{\Spc X})$, and vice versa
by taking $\psi_i=\Lop^\ast \varphi_i$. In the limit, we have that
$v=\lim_{i\to\infty} \varphi_i=\Lop_{\V \phi}^{-1\ast}\{\lim_{i\to\infty} \psi_i\}=\Lop_{\V \phi}^{-1\ast}u \in \Spc X$ and 
$u=\lim_{i\to\infty} \psi_i=\Lop^\ast\{\lim_{i\to\infty}\varphi_i\}=\Lop^\ast v \in \Spc U$. The last characterization also yields that
\begin{align}
\label{Eq:pnextended}
\langle \widetilde{p}_n,u\rangle=\lim_{i\to\infty} \langle p_n,\psi_i\rangle=\lim_{i\to\infty} \langle \Lop^\ast p_n,\varphi_i\rangle=0
\end{align}
for all $u \in \Spc U$, which is consistent with the property that $\widetilde{\Proj}_{\Spc N_\V \phi}u=0$.
The conclusion is that $ \widetilde{\V p}(u)=\V 0$ for all $u \in \Spc U$ so that the extended projector $\widetilde{\Proj}_{\Spc N_\V \phi}: \Spc X_\Lop \toC \Spc N_\V \phi$
retains
the same functional form as before as
\begin{align}
\label{Eq:ExtendProj2}
\widetilde{\Proj}_{\Spc N_\V \phi}\{g\}= \sum_{n=1}^{N_0} \langle \widetilde{p}_n, g\rangle \phi_n.
\end{align}
Due to the isometric isomorphism between $\Spc U$ and $\Spc X$, it is then also acceptable to decompose 
the extended pseudo-inverse as $\widetilde{\Lop_{\V \phi}^{-1\ast}}=\Lop^{\ast-1} (\Identity- \widetilde{\Proj}_{\Spc N_\V \phi})$,
where $\Lop^{\ast-1}$ denotes the formal inverse of $\Lop^\ast$ from $\Spc U\to \Spc X$.
By plugging in the relevant Cauchy sequences and by invoking \eqref{E:PseudoInvExt} and \eqref{Eq:pnextended}, 
it is then possible to seamlessly transfer Properties
1-3 of Proposition \ref{Prop:PreBanach} to the completed counterpart of these spaces, which yields Items 1-4.

\item {\em (iv) Embeddings}\\
We simplify the notation by setting $\Spc Y=L_{1,-\alpha}(\R^d)$ and consider the decomposition $\Spc Y=\Spc Y_{\V p^\perp} \oplus \Spc N_\V \phi$, where
$\Spc Y_{\V p^\perp}\eqdef \{\psi \in \Spc Y: \V p(\psi)=\V 0 \}$.
To prove that $\Spc Y_{\V p^\perp}\subseteq \Spc U$, we first invoke the
continuity of $\Lop_{\V \phi}^{-1\ast}: \Spc Y \toC \Spc X$, which ensures
that $v=\Lop_{\V \phi}^{-1\ast}\psi \in \Spc X$ for all $\psi \in \Spc Y_{\V p^\perp}$. We then use Property 4 (or the injectivity of $\Lop^\ast$) to get that
$\Lop^\ast v=\Lop^\ast\Lop_{\V \phi}^{-1\ast}\psi=\psi$
(because $\Proj_{\Spc N_\V \phi}\{\psi\}=0$ for all $\psi \in \Spc Y_{\V p^\perp}$), which shows that
$\psi \in \Lop^\ast(\Spc X)=\Spc U$.
This, together with the continuity of $\Lop^{\ast}: \Spc X \toC \Spc U$, ensures the continuity of the inclusion/identity map
$\Op I=\Lop^\ast \circ \Lop_{\V \phi}^{-1\ast}: \Spc Y_{\V p^\perp} \toC \Spc X \toC \Spc U$, which is equivalent to 
 $\Spc Y_{\V p^\perp} \embedC \Spc U$. Consequently, we have that $\Spc Y=(\Spc Y_{\V p^\perp} \oplus \Spc N_\V \phi) 
 \embedC
 (\Spc U \oplus \Spc N_\V \phi)=\Spc X_\Lop$. Moreover, since $\Spc P_\Lop \subseteq \Spc Y=L_{1,-\alpha}(\R^d)$ (see Item 5 in Proposition \ref{Prop:PreBanach})
 and $\Spc P_\Lop$ is a dense subspace of
 $\Spc X_\Lop=\overline{\Spc P}_\Lop$ by construction, we readily deduce that the embedding
 $\Spc Y \embedC \Spc X_\Lop$ is dense. 
Likewise, since $\Spc S(\R^d)\embedD \Spc Y$, we get that $\Spc S(\R^d) \embedD  \Spc X_\Lop$ by transitivity.
Finally, the continuity of $\Lop^\ast: \Spc X \toC \Spc S'(\R^d)$ implies that $\Spc U=\Lop^\ast(\Spc X) \embedC \Spc S'(\R^d)$ which, together with
 $\Spc N_\V \phi \embedC\Spc S'(\R^d)$, yields that $\Spc X_\Lop \embedC \Spc S'(\R^d)$.
Here too, the embedding is dense due to the property that $\Spc S(\R^d) \embedD \Spc S'(\R^d)$ (see Propositions \ref{Prop:Sdense} and \ref{Prop:Hierarchy} in Appendix A).

\item {\em (v) Identification of $\widetilde{p}_n=p_n$, $\widetilde{\Proj}_{\Spc N_\V \phi}=\Proj_{\Spc N_\V \phi}$ and $\widetilde{\Lop_{\V \phi}^{-1\ast}}={\Lop_{\V \phi}^{-1\ast}}$}\\
So far, we have distinguished the extended operators and functionals from the original ones whose initial domain was restricted to $L_{1,-\alpha}(\R^d)$. 
We now invoke the Schwartz-Banach property of both $\Spc X_\Lop$ (see Item ({\em iv}) and Definition \ref{Def:SchwartzBanach}) and $L_{1,-\alpha}(\R^d)$ 
to argue that there is a common underlying 
characterization (see Proposition \ref{Prop:Concrete}) 
that is applicable to both instances. 
Consequently, it is acceptable to 
 write that $\tilde p_n: g \mapsto \langle p_n,g\rangle$, which then gives a concrete and rigorous interpretation of the extended functionals and, by the same token, the extended projector \eqref{Eq:ExtendProj2}.
Likewise, from now on, we 
denote both the original and extended pseudo-inverse operators by 
$\Lop_{\V \phi}^{-1\ast}$, under the understanding that their underlying Schwartz kernel is the same.

 \end{proof}


An important observation is that the $\Spc X$-stability hypothesis (\ie
$\Lop_\V \phi^{-1\ast}: L_{1,-\alpha}(\R^d) \toC \Spc X$) is only required for the proof of the embeddings (last statement of the theorem). This is a fundamental
point as it ensures that $\Spc X_\Lop$ is a Schwartz-Banach space, while it also yields a concrete interpretation of the underlying operators. Last but not least, it guarantees that the actual native space
$\Spc X'_\Lop$ is a proper Banach subspace of $\Spc S'(\R^d)$ 
(by Proposition \ref{Prop:SchwartzBanach}).

Implicit in the statement of Theorem \ref{Theo:pre-dual2} (and explicit in the proof) are the following
fundamental properties of $\Spc U$: the primary part of $\Spc X_\Lop$ ``perpendicular'' to $\Spc N_\V \phi$.
\begin{corollary}
\label{Corol:U}
The space $\Spc U=\{u=\Lop^\ast v: v \in \Spc X\}$ in Theorem \ref{Theo:pre-dual2} has the following properties:
\begin{enumerate}
\item For any $u \in \Spc U$, $\Lop_\V \phi^{-1\ast}u=\Lop^{-1\ast}u \in \Spc X$ and $\Lop^\ast\Lop^{-1\ast}u=\Lop^\ast\Lop_\V \phi^{-1\ast}u=u$.
\item $\Spc U$ is a Banach space equipped with the norm $\|u\|_{\Spc U}=\|\Lop_{\V \phi}^{-1\ast}u\|_{\Spc X}$.

\item $\Spc U=\Lop^\ast(\Spc X)$ is isometrically isomorphic to $\Spc X$: For any $u \in \Spc U$ (resp. for any $v \in \Spc X$), there exists a unique element
$v=\Lop^{-1\ast}_{\V \phi}u \in \Spc X$ 
(resp. $u=\Lop^\ast vÊ \in \Spc U$) such that $\|u\|_{\Spc U}=\|v\|_{\Spc X}$.
\item $\Lop^\ast: \Spc X \toIso \Spc U$ (isometry).
\item $\Lop^{-1\ast}=\Lop_\V \phi^{-1\ast}: \Spc U \toIso \Spc X$  (isometry).
\item For any $(p,u)\in (\Spc N_\V p \times \Spc U), \langle p,u\rangle
=0$. 
\item $\Spc U$ is the completion of $\Spc S_{\V p^\perp}(\R^d)=\{\psi \in \Spc S(\R^d): \V p(\psi)=\V 0\}$ in the $\|\cdot\|_{\Spc U}$-norm.
\end{enumerate}
\end{corollary} 
\begin{proof}
Items 1-5 are re-statements/re-interpretations of the invertibility Properties 3 and 4 in Theorem  \ref{Theo:pre-dual2}. The key is that 
$(\Identity-\Proj_{\Spc N_\V \phi})\{u\}=
u$ for all $u \in \Spc U$, which then makes the presence of this operator redundant.
Item 6 results from the simple manipulation
$$ \langle p,u\rangle=\langle p, \Lop^\ast \Lop_{\V \phi}^{-1\ast}u\rangle=\langle \underbrace{\Lop p}_{=0},\Lop_{\V \phi}^{-1\ast}u\rangle
=0,$$
which is legitimate because $u \in \Spc X_\Lop=\Spc U \oplus \Spc N_\V \phi$ and $p \in \Spc N_\V p=\Spc N'_\V \phi \subset \Spc X'_\Lop=(\Spc U \oplus \Spc N_\V \phi)'$.
As for Item 7, we consider
the direct-sum decomposition $\Spc S(\R^d)= \Spc S_{\V p^\perp}(\R^d) \oplus \Spc N_\V \phi$, which is valid whenever $\Spc N_\V \phi \subset \Spc S(\R^d)$ (universality assumption). We then observe that $$(\Spc S(\R^d),\|\cdot\|_{\Spc X_\Lop})=(\Spc S_{\V p^\perp}(\R^d),\|\cdot\|_{\Spc U}) \oplus (\Spc N_\V \phi,\|\cdot\|_{\Spc N_\V \phi}).$$
This equality holds because $\|\varphi\|_{\Spc X_\Lop}=\|\psi\|_{\Spc U}+\|\V p(\phi)\|_2$ for any $\varphi=\psi+\phi \in \Spc S(\R^d) \subseteq \Spc X_\Lop$
with $(\psi,\phi) \in \big(\Spc S_{\V p^\perp}(\R^d) \times \Spc N_\V \phi\big)\subseteq (\Spc U \times \Spc N_\V \phi)$.
The denseness of the embedding $\Spc S(\R^d) \embedC \Spc X_\Lop$ from Theorem \ref{Theo:pre-dual2} implies that
$\overline{(\Spc S(\R^d),\|\cdot\|_{\Spc X_\Lop})}=\Spc X_\Lop=\Spc U \oplus \Spc N_{\V \phi}$.
Finally, by invoking Theorem \ref{Theo:DirectSum} and the property that $\overline{(\Spc N_\V \phi,\|\cdot\|_{\Spc N_\V \phi})}=\Spc N_\V \phi$ (because $\Spc N_{\V \phi}$ is finite-dimensional),
we deduce that $\overline{(\Spc S_{\V p^\perp}(\R^d),\|\cdot\|_{\Spc X_\Lop})}=\Spc U$, which is equivalent to
$\Spc S_{\V p^\perp}(\R^d) \embedD \Spc U$.
\end{proof}

Item 1 in Corollary \ref{Corol:U} indicates that the pseudo-inverse $\Lop_{\V \phi}^{-1\ast}$ and the canonical adjoint inverse
$\Lop^{-1\ast}$ are undistinguishable on $\Spc U$. This implies that the topology of $\Spc U$ does not depend on the choice of biorthogonal system $(\V \phi,\V p)$.
By contrast, the effect of the two inverse operators is very
different on $\Spc N_{\V \phi}$: for any $\phi\in \Spc N_{\V \phi}$,  $\Lop_{\V \phi}^{-1\ast}\{\phi\}=0$ by design (Property 1), while $q=\Lop^{-1\ast}\{\phi\}\in L_{\infty,\alpha}(\R^d)$ is nonzero and, in general, not even included in $\Spc X$ unless $\phi=0$.

We end this section by listing the properties of $\Spc X_\Lop$ that we believe to be the
most relevant to practice. They are directly deducible from
Theorem \ref{Theo:pre-dual2}, too.
\begin{corollary}
\label{Corol:pre-dual}
Under the admissibility and compatibility hypotheses of Definition \ref{Def:Hypotheses}, the pre-dual space
$\Spc X_\Lop$ has the following properties:
\begin{itemize}
\item $\Spc X_\Lop$ is the completion of $\Spc S(\R^d)$ in the $\|\cdot\|_{\Spc X_\Lop}$-norm, which is specified  as $\|\varphi\|_{\Spc X_\Lop}=\max(\|\Lop_{\V \phi}^{-1\ast}\varphi\|_{\Spc X},\|\V p(\varphi)\|_2)$.
\item Let $f\in \Spc S'(\R^d)$. Then, $f \in \Spc X_\Lop$ if and only if there exists $(v,\phi) \in (\Spc X \times \Spc N_\V \phi)$ such that $f=\Lop^\ast v+\phi$. Moreover, the decomposition is unique with
$v=\Lop_{\V \phi}^{-1\ast}f$ and $\phi=\Proj_{\Spc N_\V \phi}f$.
\item $\Spc X_\Lop$ is a Schwartz-Banach space.
\end{itemize}

\end{corollary}

\subsection{Native space}
\begin{figure}
\
\centerline{
\includegraphics[width=10cm]{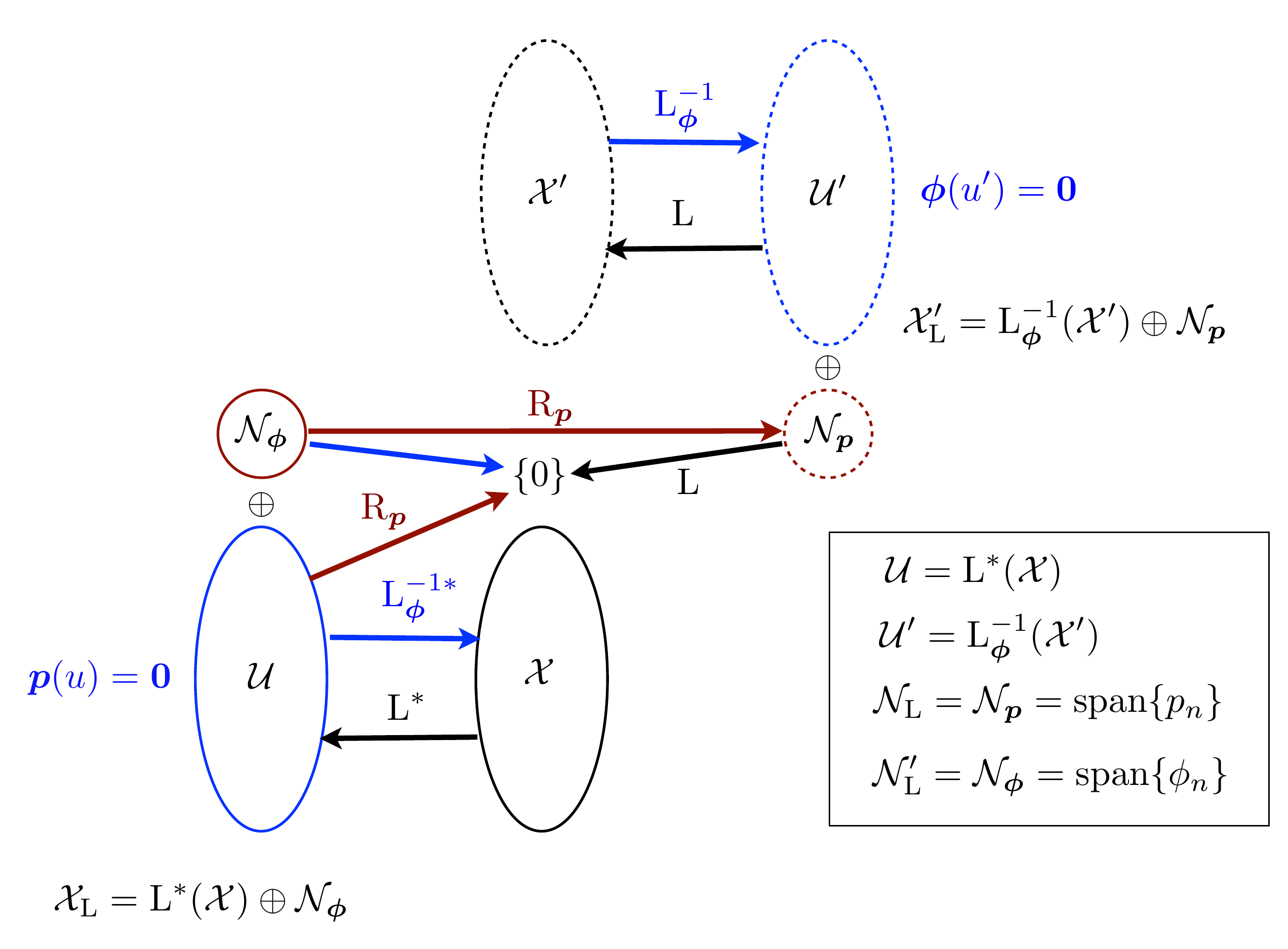}}
\caption{Schematic representation of the operators and Banach spaces that appear in the definition of $\Spc X'_\Lop$ (the native space for $\Lop$) and its pre-dual $\Spc X_{\Lop}$.
\label{Fig:Banach}
} 
\end{figure}

As indicated by the notation, the native space $\Spc X'_\Lop$ is the continuous dual of
$\Spc X_\Lop=\Spc U \oplus \Spc N_\V \phi$, where $\Spc U=\Lop^\ast(\Spc X)$. Accordingly, there is
a direct correspondence between the properties of $\Spc X_\Op L'$ and those of the predual space $\Spc X_\Lop$ in Theorem \ref{Theo:pre-dual2}.
The whole functional picture is summarized in Figure \ref{Fig:Banach}.
\begin{theorem} [Native Banach Space]
\label{Theo:NativeSpace2}
Under the admissibility and compatibility hypotheses of Definition \ref{Def:Hypotheses}, 
the continuous dual of $\Spc X_\Lop$
in Theorem \ref{Theo:pre-dual2} is the native
Banach space 
\begin{align}
\Spc X'_\Lop&=\Spc U' \oplus \Spc N_\V p=\Lop_\V \phi^{-1}(\Spc X') \oplus \Spc N_\V p \nonumber\\
&=\{\Lop^{-1}_{\V \phi} w + p: w \in \Spc X', p\in\Spc N_{\V p}\},
\end{align}
which is isometrically
isomorphic to $\Spc X' \times \Spc N_\V p$ equipped with the composite norm $\|w\|_{\Spc X'}+\|\V \phi(p)\|_2$. In other words, for any $ f\in \Spc X'_\Lop$, there is a unique pair
$w=\Lop f \in \Spc X'$ and $p=\Proj_{\Spc N_\V p}\{f\}= \sum_{n=1}^{N_0} \langle \phi_n,f\rangle p_n\in \Spc N_\V p$, with the finite-dimensional space $\Spc N_\V p= {\rm span}\{p_n\}_{n=1}^{N_0}$ being the null space of the operator $\Lop: \Spc X'_\Lop \toC \Spc X'$. 
This is consistent with the operator $\Lop_\V \phi^{-1}: \Spc X' \toC \Spc X'_\Lop$,
which is the adjoint of $\Lop_\V \phi^{-1\ast}$ in Theorem \ref{Theo:pre-dual2},
having the following properties:
\begin{enumerate}
\item Effective range: $\Spc U'= \Lop_\V \phi^{-1}(\Spc X')\eqdef \{s=\Lop_\V \phi^{-1}w: w \in \Spc X'\}$.
\item Annihilator: $\Spc N_\V \phi=\{g \in \Spc X_\Lop: \langle\Lop_\V \phi^{-1}w, g\rangle=0 \mbox{ for all } w \in \Spc X'\}$.
\item Right inverse of $\Lop$: $\Lop \Lop_\V \phi^{-1} w=w$  for any $w \in \Spc X'$.
\item Left pseudo-inverse: $ \Lop_\V \phi^{-1}\Lop f=(\Identity- \Proj_{\Spc N_\V p})\{f\}$ for any $f \in \Spc X'_\Lop$.
\end{enumerate}
Moreover, we have the hierarchy of continuous (and sometime dense) embeddings described by
$$\Spc S(\R^d) \embedC \Spc X'_\Lop \embedC L_{\infty,\alpha}(\R^d) 
\embedD \Spc S'(\R^d).$$
Finally, if $\Spc X$ is reflexive, then $\Spc X'_\Lop$ is reflexive as well and we have the dense embedding $\Spc S(\R^d) \embedD \Spc X'_\Lop $.
\end{theorem}
\begin{proof} The listed properties are the dual transpositions of the ones in Theorem \ref{Theo:pre-dual2}.
The key is $\Spc N'_{\V \phi}=\Spc N_{\V p}$ (see explanation in Section \ref{Sec:NulSpace})
and 
$\Spc X'_\Lop=(\Spc U\oplus \Spc N_\V \phi)'=\Spc U'\oplus \Spc N'_\V \phi=\Spc U'\oplus \Spc N_{\V p}$ equipped with the dual composite norm 
\begin{align*}
\|f\|_{\Spc X'_\Lop}&=\|(\|\Proj_{\Spc U}f \|_{\Spc U'},\|\Proj_{\Spc N_{\V p}}f\|_{\Spc N_\V p})\|_1\\
&=\|\Proj_{\Spc U'}f \|_{\Spc U'} +\|\V \phi(f)\|_2.
\end{align*}
For the details, the reader is referred to Appendix B on direct sums and Proposition \ref{Prop:dualDirectSum} with $(p,q)=(\infty,1)$.
 The other fundamental ingredient is the continuity of the adjoint operators
 $\Lop: \Spc U' \toIso \Spc X'$, $\Lop_{\V \phi}^{-1}: \Spc X' \toIso \Spc U'$,
 $\Lop_{\V \phi}^{-1}: \Spc X' \toC \Spc X'_\Lop$, $\Proj^\ast_{\Spc N_\V \phi}=\Proj_{\Spc N_\V p}: \Spc X'_\Lop \toC \Spc N_\V p \embedIso \Spc X'_\Lop$,
 which follows from the continuity of $\Lop^\ast: \Spc X \toIso \Spc U$, $\Lop_{\V \phi}^{-1\ast}: \Spc U \toIso \Spc X$, $\Lop_{\V \phi}^{-1\ast}: \Spc X_\Lop \toC \Spc X$, $\Proj_{\Spc N_\V \phi}: \Spc X_\Lop \toC\Spc N_\V \phi \embedIso \Spc X_\Lop$ in Theorem \ref{Theo:pre-dual2} and Corollary \ref{Corol:U}.
 The adjoint relation $\Proj_{\Spc N_\V p}=\Proj^\ast_{\Spc N_\V \phi}$ is due to the special form of the underlying kernel (see Table 1).
 
 \item  {\em (i) Identification of $\Spc U' =\Lop_{\V \phi}^{-1}(\Spc X')$ with $\|s\|_{\Spc U'}=\|\Lop s\|_{\Spc X'}$}\\
Since the mapping between $\Spc X'$ and $\Spc U'$ is isometric and bijective,
we have that $\Spc U' =\Lop_{\V \phi}^{-1}(\Spc X')$, with the two spaces being isometrically isomorphic.
By recalling the definition of the dual norm and invoking the isomorphism between $\Spc X$ and $\Spc U=\Lop^\ast(\Spc X)$ with $v \mapsto u=\Lop^\ast v$, we then get that
\begin{align*}
\|s\|_{\Spc U'}&=\sup_{u \in \Spc U \backslash \{0\}} \frac{\langle s, u \rangle_{\Spc U' \times \Spc U}Ê}{\|u\|_{\Spc U}Ê}=\sup_{v \in \Spc X\backslash \{0\}} \frac{\langle s, \Lop^\ast v \rangle_{\Spc U' \times \Spc U}ÊÊ}{\|v\|_{\Spc X}Ê}\\
&=\sup_{v \in \Spc X\backslash \{0\}} \frac{\langle \Lop s, v \rangle_{\Spc X' \times \Spc X}}{\|v\|_{\Spc X}Ê}=\|\Lop s\|_{\Spc X'}.
\end{align*}

 \item  {\em (ii) Derivation of Properties 2-4 by duality}\\
 The underlying principle is that the weak topology (resp. the weak$^\ast$ topology) separates the points in (resp. the dual of) a locally convex vector space
\cite{Rudin1991}.
Specifically, let $(\Spc X', \Spc X)$ be any dual pair of Banach spaces. Then, for any $g_1, g_2 \in \Spc X$ and $f_1,f_2 \in \Spc X'$, 
\begin{align*}
g_1=g_2 \quad &\Leftrightarrow \quad \langle f, g_1 \rangle_{\Spc X'\times \Spc X}=\langle f, g_2 \rangle_{\Spc X'\times \Spc X} \mbox{ for all } f\in \Spc X',\\
f_1=f_2 \quad &\Leftrightarrow \quad \langle f_1, g \rangle_{\Spc X'\times \Spc X}= \langle f_2, g \rangle_{\Spc X'\times \Spc X} \mbox{ for all } g\in \Spc X.
\end{align*}
Consequently, the null-space Property 2 in Theorem \ref{Theo:pre-dual2} is equivalent to
$$
g\in \Spc N_\V \phi \ \Leftrightarrow\ 0=\langle w, \Lop_{\V \phi}^{-1\ast} g\rangle_{\Spc X' \times \Spc X}=\langle \Lop_{\V \phi}^{-1}w, g\rangle_{\Spc X'_\Lop \times \Spc X_\Lop}, \forall w \in \Spc X', $$
which is the desired result. 
The same principle applies for the other properties.


\item {\em (iii) Embeddings}\\
The application of Theorem \ref{Theo:Dense} to the series of continuous and dense embeddings in
Theorem \ref{Theo:pre-dual2}
yields $$\Spc S(\R^d) \embedC\Spc X'_\Lop \embedC \big(L_{1,-\alpha}(\R^d)\big)'=L_{\infty,\alpha}(\R^d)\embedC\Spc S'(\R^d).$$
The denseness of the embedding $\Spc X'_\Lop \embedD \Spc S'(\R^d)$ follows from $\Spc S(\R^d) \embedD  \Spc S'(\R^d)$ and Proposition \ref{Prop:Hierarchy}.

\end{proof}

We note that Property 2 has the other equivalent formulation
$$\Spc U'=\{g \in \Spc X'_\Lop : \V \phi(g)=\V 0\}=\{g\in \Spc X'_\Lop: \Proj_{\Spc N_\V p}\{g\}=0\},
$$
which is consistent with the direct-sum property.
Hence the combination of Properties 2-4 implies a perfect isometry between $\Spc U'$ and $\Spc X'$
with $w=\Lop s \in \Spc X'$, $s=\Lop_\V \phi^{-1}w \in \Spc U'$,  and $\|s\|_{\Spc U'}=\|\Lop s\|_{\Spc X'}$.
The consideration of the direct-sum decomposition
$f=s+p$, where $p=\Proj_{\Spc N_\V p}\{f\} \in \Spc N_\V p$ and $s=(f-p)=\Lop_\V \phi^{-1}  w\in \Spc U'$,
allows us to identify the norm of $\Spc X'_\Lop$ as
\begin{align}
\|f\|_{\Spc X'_\Lop}&=\|s\|_{\Spc U'}+\|\V \phi (p)\|_2 =\|\Lop s\|_{\Spc X'}+\|\V \phi (p)\|_2 \nonumber \\
&=\|\Lop f\|_{\Spc X'}+
\|\V \phi(f)\|_2,
\end{align}
where we have made use of the property that $\Lop\{s + p\}=\Lop s$ and $\V \phi (s)=\V 0$
for all $(s,p) \in (\Spc U' \times \Spc N_\V p)$.
This is the basis for a restatement of the primary properties of $\Spc X'_\Lop$ in Theorem \ref{Theo:NativeSpace2} in a form 
more suitable for practitioners.

\begin{corollary}
\label{Corol:NativeSpace}
Under the admissibility and compatibility hypotheses of Definition \ref{Def:Hypotheses},  the native space of $(\Lop, \Spc X')$, denoted by $\Spc X'_\Lop$, has the following properties:
\begin{itemize}
\item $\Spc X'_\Lop$ is a Banach space that admits the explicit definition
\begin{align}\label{NativeDefRigorous}
\Spc X'_\Lop=\{f \in \Spc S'(\R^d): \|f\|_{\Spc X'_\Lop}= \sup_{\|\varphi\|_{\Spc X_\Lop}\le 1: \ \varphi \in \Spc S(\R^d)} \langle f, \varphi\rangle<\infty\}
\end{align}
or, equivalently,
$$
\Spc X'_\Lop=\{f \in L_{\infty,\alpha}(\R^d): \|f\|_{\Spc X'_\Lop}<\infty\},
$$
where $\alpha$ is the growth order associated with $\Lop$.
Moreover,
$f\in \Spc X'_\Lop \Leftrightarrow \|f\|_{\Spc X'_\Lop}=\|\Lop f\|_{\Spc X'}+\|\V \phi(f)\|_2$.
In particular, $f\in \Spc X'_\Lop \Rightarrow \|\Lop f\|_{\Spc X'}<+\infty$.
\item Let $f \in \Spc S'(\R^d)$. Then, $f \in \Spc X'_\Lop$ if and only if there exists $(w,p) \in (\Spc X' \times \Spc N_\V p)$ such that $f=\Lop_{\V \phi}^{-1}w+p$. Moreover, the decomposition is unique, with
$w=\Lop f$ and $p=\Proj_{\Spc N_\V p}\{f\}$.
\end{itemize}

\end{corollary}

Equation \eqref{NativeDefRigorous}
provides a rigorous, self-contained definition of $\Spc X'_\Lop$, but it has the disadvantage of being a bit convoluted. 
We like to view this equation as the justification for the two alternative
forms
\begin{align}
\Spc X'_\Lop&=\left\{f \in L_{\infty,\alpha}(\R^d): \|f\|_{\Spc X'_\Lop}=\|\Lop f\|_{\Spc X'}+\|\V \phi(f)\|_2 
<\infty\right\} \label{Eq:Def2Native}\\
\Spc X'_\Lop&=\left\{f \in L_{\infty,\alpha}(\R^d): \|\Lop f\|_{\Spc X'}<\infty\right\}, \label{Eq:Def3Native}
\end{align}
which are adequate if one implicitly assumes that $\Spc N_\V \phi\subset L_{1,-\alpha}(\R^d)$ and $f \notin \Spc X'_\Lop \Leftrightarrow \|\Lop f\|_{\Spc X'}+\|\V \phi(f)\|_2=\infty$.
While $\Lop f$ is {\em a priori} undefined for $f \notin \Spc X'_\Lop$, one circumvents the formal difficulty by adopting the more permissive dual definition of the underlying semi-norm
\begin{align*}
\sup_{\|\varphi\|_{\Spc X}\le 1: \ \varphi \in \Spc S(\R^d)}
\langle f, \Lop^\ast \varphi \rangle=\left\{ \begin{array}{ll}
\|\Lop f\|_{\Spc X'},  &   f \in \Spc X'_\Lop   \\
+\infty, &   \mbox{ otherwise},
 \end{array}
\right.
\end{align*}
which is valid for any $f\in L_{\infty,\alpha}(\R^d)$ in reason of the denseness of $\Spc S(\R^d)$ in $\Spc X$ and the admissibility condition $\Lop^\ast \varphi \in L_{1,-\alpha}(\R^d)$.
Under this interpretation, \eqref{Eq:Def3Native} is legitimate as well since $\Spc X'_\Lop\embedC L_{\infty,\alpha}(\R^d)$. Indeed, the hypothesis of spline-admissibility (Definition \ref{Def:splineadmis3})
requires that the growth-restricted null space of $\Lop$ be finite-dimensional and spanned by some basis  $\V p=(p_1,\dots,p_{N_0})$, while one necessarily has that
$\|\V \phi(f)\|_2<\infty$ for all $f \in L_{\infty,\alpha}(\R^d)$ because $\Spc N_\V \phi \subset L_{1,-\alpha}(\R^d) \embedIso \big(L_{\infty,\alpha}(\R^d)\big)'=\big(L_{1,-\alpha}(\R^d)\big)''$. The slight disadvantage is that \eqref{Eq:Def3Native} does not fully specify the underlying topology.

\subsection{Equivalent topologies and biorthogonal systems}
In order to show that the choice of one biorthogonal system over another---say, $(\tilde{\V \phi}, \tilde{\V p})$ vs.\  $(\V \phi, \V p)$---has no direct incidence on the definition of the underlying native
space, 
we start by extending the range of validity of Theorems \ref{Theo:pre-dual2} and \ref{Theo:NativeSpace2} 
to the complete set of admissible systems $(\tilde{\V \phi}, \tilde{\V p})$ with
$\tilde p_1, \dots,\tilde p_{N_0} \in \Spc N_{\Lop}$, $\tilde \phi_1, \dots,\tilde \phi_{N_0} \in \Spc X_{\Lop}$, and 
$[\tilde{\V \phi}(\tilde p_1) \cdots \tilde{\V \phi}(\tilde p_{N_0}) ]=\M I$ (biorthogonality).
To that end,  we rely on the existence of a primary biorthogonal system
such that $ \Spc N_{{\V \phi}}={\rm span}\{\phi_n\}\subset \Spc S(\R^d)$ (universality property), which ensures that
the initial pre-dual space $\Spc X_\Lop$ is well-defined.
\begin{proposition}
\label{Prop:ExtensionTheorem}
Let $\tilde {\V \phi}=(\tilde\phi_n)$ with $\tilde \phi_n \in \Spc X_\Lop$ for $n=1,\dots,N_0$ be such that the matrix $\M C=[\tilde{\V \phi}(p_1) \cdots \tilde{\V \phi}(p_{N_0}) ]\in \R^{N_0 \times N_0}$
is invertible. Then, there exists a unique basis $\tilde {\V p}$ of $\Spc N_\Lop=\Spc N_{\tilde{\V p}}$ such that Theorems 
\ref{Theo:pre-dual2}-\ref{Theo:NativeSpace2} and Corollaries \ref{Corol:U}-\ref{Corol:NativeSpace}
 remain valid for the biorthogonal system $(\tilde {\V \phi},\tilde {\V p})$ and define a pair of native and pre-dual spaces 
 $\tilde{\Spc X}'_{\Lop}=\Spc U' \oplus  \Spc N_{\tilde{\V p}}$ and $\tilde{\Spc X}_{\Lop}=\Spc U \oplus \Spc N_{\tilde{\V \phi}}$,
 where $\Spc U=\Lop^\ast(\Spc X)$. This 
 construction then specifies four operators
 with the following properties:
 \begin{itemize}
 \item $ \Proj_{\Spc N_{\tilde{\V \phi}}}=\Proj^\ast_{\Spc N_{\tilde{\V p}}}: \tilde{\Spc X}_{\Lop} \toC \Spc N_{\tilde{\V \phi}}: 
 g \mapsto \sum_{n=1}^{N_0} \tilde \phi_n \langle \tilde p_n, g \rangle $ such that
\begin{align*}
\forall
 \phi \in \Spc N_{\tilde{\V \phi}}: \quad &\Proj_{\Spc N_{\tilde{\V \phi}}}\{\phi\}=\phi\\
\forall
u \in \Spc U: \quad &\Proj_{\Spc N_{\tilde{\V \phi}}}\{u\}=0.
\end{align*}
 \item $\Lop_{\tilde{\V \phi}}^{-1\ast}=\Lop^{-1\ast}(\Identity-\Proj_{\Spc N_{\tilde{\V \phi}}}): L_{1,-\alpha}(\R^d) \embedC \tilde{\Spc X}_{\Lop} \toC \Spc X$
  such that
  \begin{align*}
\forall
\phi \in \Spc N_{\tilde{\V \phi}}: \quad &\Lop_{\tilde{\V \phi}}^{-1\ast}\phi=0\\
\forall
g \in \tilde{\Spc X}_\Lop: \quad &\Lop^\ast \Lop_{\tilde{\V \phi}}^{-1\ast}g=(\Identity-\Proj_{\Spc N_{\tilde{\V \phi}}})\{g\}\\
\forall
v \in \Spc X: \quad &\Lop_{\tilde{\V \phi}}^{-1\ast}\Lop^\ast v=v.
\end{align*}
  
  \item $ \Proj_{\Spc N_{\tilde{\V p}}}: \tilde{\Spc X}'_{\Lop} \toC \Spc N_{\tilde{\V p}}: 
 f \mapsto \sum_{n=1}^{N_0}  \tilde p_n\langle \tilde \phi_n, f \rangle$ such that
\begin{align*}
\forall
p\in \Spc N_{\tilde{\V p}}: \quad &\Proj_{\Spc N_{\tilde{\V p}}}\{p\}=p\\
\forall
v \in \Spc U': \quad &\Proj_{\Spc N_{\tilde{\V p}}}\{v\}=0.
\end{align*}
 
 \item $\Lop_{\tilde{\V \phi}}^{-1}: \Spc X' \toC \tilde{\Spc X}_\Lop'$
such that
\begin{align*}
\forall
w \in \Spc X': \quad &\tilde{\V \phi}(\Lop_{\tilde{\V \phi}}^{-1}\{w\})=\V 0\\
\forall
w \in \Spc X': \quad &\Lop\Lop_{\tilde{\V \phi}}^{-1}w=w\\
\forall
f \in \Spc X_\Lop': \quad & \Lop_{\tilde{\V \phi}}^{-1}\Lop\{f\}=(\Identity-\Proj_{\Spc N_{\tilde{\V p}}})\{f\}.
\end{align*}

\end{itemize}

\end{proposition} 
\begin{proof}
The new basis
$\tilde{\V p}$ of $\Spc N_{\V p}=\Spc N_{\tilde{\V p}}$ is given by
$
\tilde{\V p}=\M C^{-1} \V p
$, which can easily be seen to be biorthogonal to $\tilde{\V \phi}$.
Next, we check that the operators $\Proj_{\Spc N_{\tilde{\V \phi}}}$ and
$(\Identity-\Proj_{\Spc N_{\tilde{\V \phi}}})$ are continuous on $\Spc X_\Lop=\Lop^\ast(\Spc X) \oplus \Spc N_{\V p}$ under the hypothesis that $\tilde \phi_n \in \Spc X_\Lop$.
Specifically, for any $g \in \Spc X_\Lop$, we have that
\begin{align*}
\|\Proj_{\Spc N_{\tilde{\V \phi}}}g\|_{\Spc X_\Lop} &\le \sum_{n=1}^{N_0} |\langle \tilde p_n, g \rangle| \; \|\tilde \phi_n\|_{\Spc X_\Lop} \tag{by the triangle inequality}\\
 & \le  \sum_{n=1}^{N_0} \|\tilde p_n\|_{\Spc X'_\Lop} \|g\|_{\Spc X_\Lop} \|\tilde \phi_n\|_{\Spc X_\Lop} = C_1 \|g\|_{\Spc X_\Lop},
\end{align*}
where $C_1=\sum_{n=1}^{N_0} \|\tilde p_n\|_{\Spc X'_\Lop}  \|\tilde \phi_n\|_{\Spc X_\Lop}<\infty$. 
The property that $\Proj_{\Spc N_{\tilde{\V \phi}}}\{g\} \in \Spc N_{\tilde{\V \phi}}$ (by construction) translates into
$\Proj_{\Spc N_{\tilde{\V \phi}}}: \Spc X_\Lop \toC \Spc N_{\tilde{\V \phi}}\embedC \Spc X_\Lop$. It is also obvious from the definition that
$\Proj_{\Spc N_{\tilde{\V \phi}}}$ is the adjoint of $\Proj_{\Spc N_{\tilde{\V p}}}$ whose Schwartz kernel is $(\V x,\V y)\mapsto \sum_{n=1}^{N_0} \tilde p_n(\V x)\tilde \phi_n(\V y)$.

Likewise, we have that
$\|(\Identity-\Proj_{\Spc N_{\tilde{\V \phi}}})\{g\}\|_{\Spc X_\Lop}\le (1+C_1) \|g\|_{\Spc X_\Lop}$.
By invoking the 
biorthogonality of $(\tilde{\V \phi}, \tilde{\V  p})$, we readily verify that ${\tilde{\V  p}(g-\Proj_{\Spc N_{\tilde{\V \phi}}} \{g\})}=\V 0 \Leftrightarrow \V  p(g-\Proj_{\Spc N_{\tilde{\V \phi}}} \{g\})=\V 0$, which yields $(\Identity-\Proj_{\Spc N_{\tilde{\V \phi}}}) \{g\} \in \Spc U$, thereby proving that
$(\Identity-\Proj_{\Spc N_{\tilde{\V \phi}}}): \Spc X_\Lop \toC \Spc U\embedC \Spc X_\Lop$.
Since $\Lop^{-1\ast}: \Spc U \toC \Spc X$ (see Corollary \ref{Corol:U}), 
we can therefore chain the two operators, which results in $\Lop^{-1\ast}\circ (\Identity-\Proj_{\Spc N_{\tilde{\V \phi}}}): \Spc X_\Lop \toC \Spc U \toC \Spc X$,
thereby proving the continuity of $\Lop_{\tilde{\V \phi}}^{-1\ast}
: \Spc X_\Lop \toC \Spc X$.
Finally, we invoke the continuous embedding $L_{1,-\alpha}(\R^d) \embedC \Spc X_\Lop$ (see Theorem \ref{Theo:pre-dual2}),
which ensures that $\Lop_{\tilde{\V \phi}}^{-1\ast}: L_{1,-\alpha}(\R^d) \toC \Spc X$, in accordance with the last compatibility requirement in Definition \ref{Def:Hypotheses}. 

Given that the underlying operators all satisfy the required continuity and annihilation properties, we can then revisit the proofs and constructions in Theorems 
\ref{Theo:pre-dual2} and \ref{Theo:NativeSpace2} 
to specify
the corresponding pair of spaces $\tilde {\Spc X}_\Lop$ and $\tilde {\Spc X}'_\Lop$, which inherit the same embedding properties as ${\Spc X}_\Lop$ and ${\Spc X}'_\Lop$. 

\end{proof}

An important outcome of the proof of Proposition \ref{Prop:ExtensionTheorem} is that the compatibility condition for a single instance $({\V \phi}, {\V p})$ is transferred
to all admissible biorthogonal systems $(\tilde{\V \phi}, \tilde{\V p})$.
Theorem \ref{TheoNormEquiv} describes the effect of such a change of biorthogonal system 
on the underlying norms, while it ensures that
the underlying topologies are equivalent.

\begin{theorem}[Equivalent direct-sum topologies]
\label{TheoNormEquiv}
Let $\Lop$ be a spline-admissible operator that is compatible with $\Spc X'$ in the sense 
of Definition \ref{Def:Hypotheses}.
Then, for any two biorthogonal systems $(\V \phi, \V p)$ and $(\tilde{\V \phi}, \tilde{\V p})$
with $\Spc N_\V p=\Spc N_{\tilde{\V p}}=\Spc N_\Lop$ and
$\V \phi, \tilde{\V \phi} \in \Spc X_\Lop^{N_0}$, we have that
\begin{itemize}
\item $\Spc X_\Lop=\tilde{\Spc X}_{\Lop}$ and $\Spc X'_\Lop=\tilde{\Spc X}'_{\Lop}$ as sets;
\item the norms $\|\cdot\|_{\Spc X_\Lop}$ and $\|\cdot\|_{\tilde{\Spc X}_{\Lop}}$ are equivalent on $\Spc X_\Lop$;
\item the norms $\|\cdot\|_{\Spc X'_\Lop}$ and $\|\cdot\|_{\tilde{\Spc X}'_{\Lop}}$ are equivalent on $\Spc X'_\Lop$.
\end{itemize}
More precisely, with the operator definitions in Proposition \ref{Prop:ExtensionTheorem} and
\begin{align*}
\|f\|_{\tilde{\Spc X}'_{\Lop}}&= \|\Lop f\|_{\Spc X'} + \|\tilde{\V \phi}(f)\|_2\\
\|g\|_{\tilde{\Spc X}_{\Lop}}&= \|\Lop_{\tilde{\V \phi}}^{-1\ast} g\|_{\Spc X} + \|\tilde{\V p}(g)\|_2,
\end{align*}
we have the equivalence relations
\begin{align}
\forall f\in \Spc X'_\Lop:\quad &\|\Lop f \|_{\Spc X'}=\|(\Identity- \Proj_{\Spc N_{\tilde{\V p}}})\{f\} \|_{\tilde{\Spc X'_\Lop}} \label{Eq:norm1}\\
\forall f \in \Spc X'_\Lop: \quad&A_{1} \, \| f \|_{\Spc X'_\Lop} \le  \| f \|_{\tilde{\Spc X}'_{\Lop}} \le A_{2}\; \| f \|_{\Spc X'_\Lop}\label{Eq:norm2}\\
\forall p \in \Spc N_{\Lop}=\Spc N_{\V p}: \quad&B_{1} \, \|\V \phi(p)\|_2 \le \|\tilde{\V \phi}(p)\|_2 \le B_2\,  \|\V \phi(p)\|_2
\label{Eq:norm3}
\end{align}
\begin{align}
\forall u\in \Spc U=\Lop^\ast(\Spc X):\quad &\|\Lop_{\tilde{\V \phi}}^{-1\ast} u\|_{\Spc X}=\|(\Identity- \Proj_{\Spc N_{\tilde{\V \phi}}})\{u\} \|_{\Spc X_\Lop}=\|\Lop^{-1\ast} u \|_{\Spc X}
\label{Eq:norm1d}\\
\forall g \in \Spc X_{\Lop}: \quad&A'_1\, \|g \|_{\Spc X_\Lop} \le  \| g\|_{\tilde{\Spc X}_{\Lop}} \le A'_2\; \| g\|_{\Spc X_\Lop}
\label{Eq:norm2d}\\
\forall g \in \Spc X_{\Lop}: \quad&B'_1 \, \|\V p(g)\|_2 \le \|\tilde{\V p}(g)\|_2 \le B'_2\,  \|\V p(g)\|_2\label{Eq:norm3d}
\end{align}
for some suitable constants $A_1, A_2, B_1, B_2, A'_1, A'_2, B'_1, B'_2 > 0$.
\end{theorem}
Note that the fixed parts of the construction are
$\Spc N_\Lop=\Spc N_\V p=\Spc N_{\tilde{\V p}}$ and $\Spc U=\Lop^\ast(\Spc X)$ (see Figure \ref{Fig:Banach}), which are associated with \eqref{Eq:norm3} and \eqref{Eq:norm1d}, respectively.
On the other hand, we may have that $\Spc N_{\V \phi}\neq \Spc N_{\tilde{\V \phi}}$, 
although what distinguishes those two spaces needs to be included in
$\Spc U=\Lop^\ast(\Spc X)$ in the sense that $(\Identity-\Proj_{ \Spc N_{\tilde{\V \phi}} }) \{\phi\},(\Identity-\Proj_{ \Spc N_{{\V \phi}} }) \{\tilde\phi\}\in \Spc U$ for any $\phi \in 
\Spc N_{{\V \phi}}$ and $\widetilde \phi \in 
\Spc N_{\tilde{\V \phi}}$.
In particular, when
both $\V \phi$ and $\widetilde{\V \phi}$ are biorthogonal to the same $\V p=\tilde{\V p}$, we have that $\V p( \widetilde{\phi}_n-\phi_n)=\V 0$ for $n=1,\dots,N_0$, so that the
condition $\widetilde \phi_n \in \Spc X_\Lop$ is equivalent to
\begin{align}
\label{Eq:ExtendXstability}
\Lop_\V \phi^{-1\ast}\{\widetilde \phi_n-\phi_n\}=\Lop^{-1\ast}\{\widetilde \phi_n-\phi_n\} \in \Spc X,
\end{align}
which yields a criterion for admissibility that is 
convenient since it no longer depends on $\Lop_\V \phi^{-1\ast}$.

\begin{proof} Proposition \ref{Prop:ExtensionTheorem} ensures that underlying spaces and operators are well-defined.

\item {\em (i) Delineation as sets}\\ To avoid circularity, we assume once more that
$\Spc N_{{\V \phi}}={\rm span}\{\phi_n\}\subset \Spc S(\R^d)$ (universality condition).
We then consider the generic members $f=\Lop_{\V \phi}^{-1}w + p$ and $\tilde f=\Lop_{\tilde{\V \phi}}^{-1}w + p$ of the underlying spaces with $w\in\Spc X$ and $p \in \Spc N_{\V p}=\Spc N_{\tilde{\V p}}={\rm span}\{\tilde{p}_n\}$.
Since $\Lop\{\tilde{f}-f\}=(w-w)=0$, the two functions can only differ by a components $\tilde p=(\tilde{f}-f)\in \Spc N_{\V p}$,
which proves that $\Spc X'_\Lop=\tilde{\Spc X}'_{\Lop}$ (as a set).  
Likewise, on the side of the pre-dual space, we have that $g=\Lop^{\ast}v+\phi\in \Spc X_\Lop$ with $(v,\phi)\in (\Spc X \times \Spc N_{\V \phi})$ and
$ \Spc N_{{\V \phi}}\subset \Spc S(\R^d)$.
Now, if  $\Spc N_{\tilde{\V \phi}} 
\subset \Spc X_\Lop$, we obviously also have that $\tilde g=\Lop^{\ast}v+\tilde \phi \in \Spc X_\Lop$ for any $\tilde \phi\in \Spc N_{\tilde{\V \phi}}$ and $v \in \Spc X$,
which implies that $\tilde{\Spc X}_\Lop \subseteq \Spc X_\Lop$. Conversely, since
 $\Spc S(\R^d) \embedC \tilde{\Spc X}_\Lop$, $g=\Lop^{\ast}v+\phi \in \tilde{\Spc X}_\Lop$
 for any $(v,\phi)\in (\Spc X \times \Spc N_{\V \phi})$, 
 which yields $\Spc X_\Lop\subseteq \tilde{\Spc X}_{\Lop}$.

\item {\em (ii) Norm inequalities}\\
To get
\eqref{Eq:norm1}, we observe that
$$
\|(\Identity- \Proj_{\Spc N_{\tilde{\V p}}})\{f\} \|_{\tilde{\Spc X'_\Lop}} =\underbrace{ \|\Lop\big(f-\Proj_{\Spc N_{\tilde{\V p}}}\{f\}\big)\|_{\Spc X'}}_{\|\Lop f \|_{\Spc X'}}+\underbrace{\|\tilde{\V \phi}\big(f-\Proj_{\Spc N_{\tilde{\V p}}}\{f\}\big)\|_2}_{=0}
$$
because $\Proj_{\Spc N_{\tilde{\V p}}}\{f\}\in \Spc N_{\tilde{\V p}}$
is annihilated by $\Lop$ and $\tilde{\V \phi}\big(f-\Proj_{\Spc N_{\tilde{\V p}}}\{f\}\big)=\M 0$ by construction, irrespective of the choice of $(\tilde{\V \phi},\tilde{\V p})$. 
Likewise, the dual relation \eqref{Eq:norm1d} is a direct consequence of the form of the adjoint operator $\Op L^{-1\ast}_{\V \phi}=\Lop^{-1\ast}(\Identity-\Proj_{\Spc N_{\V \phi}})$ and the orthogonality condition
$\Proj_{\Spc N_{\V \phi}}\{u\}=0$ for all $u\in \Spc U$.

In order to estimate $\|f\|_{\tilde{\Spc X}'_{\Lop}}$, we first invoke the duality bound
$$
|\langle \tilde{\phi}_n, f\rangle|\le \|\tilde{\phi}_n\|_{\Spc X_\Lop}\; \|f\|_{\Spc X'_\Lop}
$$
with the role of $\Spc X_\Lop$ and $\tilde{\Spc X}_{\Lop}$ (resp. $\tilde{\phi}_n$ and $\phi_n$) being interchangeable.
This suggests the estimate
\begin{align}
 \|f\|_{\tilde{\Spc X}'_{\Lop}}&=\|\Lop f\|_{\Spc X'} +\left(\sum_{n=1}^{N_0}|\langle \tilde{\phi}_n,f\rangle|^2\right)^\frac{1}{2}
 \nonumber \\
 &\le \|\Lop f\|_{\Spc X'}  +\sum_{n=1}^{N_0}\|\tilde{\phi}_n\|_{\Spc X_\Lop}\|f\|_{\Spc X'_\Lop} \tag{by the triangle inequality} \nonumber \\
 & \le \left(1+\sum_{n=1}^{N_0}\|\tilde{\phi}_n\|_{\Spc X_\Lop}\right) \| f\|_{\Spc X'_\Lop}.
 \label{Eq:Normeqaux}
 \end{align}
Likewise, we have that
\begin{align*}
 \|f\|_{\Spc X'_\Lop} \le \left(1+\sum_{n=1}^{N_0}\|\phi_n\|_{\tilde{\Spc X}_{\Lop}}\right) \|f\|_{\tilde{\Spc X}'_{\Lop}}
 \end{align*}
 which, when combined with \eqref{Eq:Normeqaux}, yields \eqref{Eq:norm2} with
 \begin{align}
A_1=
\left(1+\sum_{n=1}^{N_0}\|\phi_n\|_{\tilde{\Spc X}_{\Lop}} \right)^{-1}
\quad  \mbox{ and } \quad A_2=\left(1+\sum_{n=1}^{N_0} \|\tilde{\phi}_n\|_{\Spc X_\Lop} \right).
\label{Eq:A1A2}
\end{align}

We apply a similar technique to derive \eqref{Eq:norm2d} by considering the generic element $g=\Lop^\ast v+ \phi \in \Spc X_{\Lop}$
with $v\in \Spc X$ and $\phi \in \Spc N_{\V \phi}$. It leads to
\begin{align*}
 \|g\|_{\tilde{\Spc X}_{\Lop}}=\|\Lop^\ast v+ \phi\|_{\tilde{\Spc X}_{\Lop}}&=\max(\|v\|_{\Spc X},\left(\sum_{n=1}^{N_0}|\langle \tilde{p}_n,\phi\rangle|^2\right)^\frac{1}{2})
 \\
 &\le \max(\|v\|_{\Spc X}  ,\sum_{n=1}^{N_0}\|\tilde{p}_n\|_{\Spc X'_\Lop}\|\phi\|_{\Spc X_\Lop} )
 \\
& \le \underbrace{\left(1+\sum_{n=1}^{N_0}\|\tilde{p}_n\|_{\Spc X'_\Lop}\right)}_{A_2'} \| g\|_{\Spc X_\Lop},
 \end{align*}
 where we have used the property that $\|g\|_{\Spc X_\Lop}=\max(\|v\|_{\Spc X},\|\phi\|_{\Spc X_\Lop})$.
Likewise, the complementary bounding constant $A_1'$
is obtained by substituting $\phi_n$ by $p_n$ and $\tilde{\Spc X}_{\Lop}$ by $\tilde{\Spc X}'_{\Lop}$ in the first part of \eqref{Eq:A1A2}.

As for the null-space component $p\in \Spc N_\V p$, we recall that any admissible $\tilde{\V \phi}$ must be such that
the cross-product matrix
$$
\M C 
=[ \tilde{\V \phi}( p_1) \cdots \tilde{\V \phi}(p_{N_0})] \in \R^{N_0\times N_0}
$$
has full rank for any basis $\V p$ of $\Spc N_{\V p}$. The biorthogonal basis
$\tilde{\V p}$ is then given by
$$
\tilde{\V p}=\M B \V p,
$$
where $\M B=\M C^{-1}$. The entries of these matrices are denoted by
\begin{align}
b_{m,n}&=[\M B]_{m,n}=\langle \phi_m, \tilde p_n\rangle\label{Eq:cmn}\\
c_{m,n}&=[\M C]_{m,n}=\langle \tilde \phi_m, p_n \rangle,
\end{align}
respectively, where the right-hand side of \eqref{Eq:cmn} follows from the
biorthogonality of $(\V \phi,\V p)$. Let us now consider some arbitrary $p=\sum_{n=1}^{N_0} \langle \phi_n, p\rangle p_n \in \Spc N_{\V p}$ whose
initial norm is $\|\V \phi(p)\|_2$. As we change the system of coordinates, we get
\begin{align*}
\|\tilde{\V \phi}(p)\|_2^2&=\sum_{m=1}^{N_0} \left|\sum_{n=1}^{N_0}\langle \phi_n, p\rangle \langle \tilde \phi_m, p_n\rangle\right|^2\\
&\le \sum_{m=1}^{N_0}\left(\sum_{n=1}^{N_0} |\langle \tilde \phi_m, p_n\rangle|^2\right)\left( \sum_{n=1}^{N_0} |\langle \phi_n, p\rangle|^2\right)\tag{by Cauchy-Schwarz}\\
& = \left(\sum_{m=1}^{N_0}\sum_{n=1}^{N_0} c_{m,n}^2\right) \|\V \phi(p)\|^2_2.
\end{align*}
Likewise, by interchanging the role of ${\V \phi}$ and $\tilde{\V \phi}$, we find that
\begin{align*}
\|{\V \phi}(p)\|_2^2&\le  \left(\sum_{m=1}^{N_0}\sum_{n=1}^{N_0} b_{m,n}^2\right) \|\tilde{\V \phi}(p)\|^2_2.
\end{align*}
The combination of these two inequalities yields \eqref{Eq:norm3} with
$B_2=\|\M C\|_F$ (the Frobenius norm of the matrix $\M C$) and $B_1=1/\|\M B\|_F$.

%
Similarly, we establish the norm inequality \eqref{Eq:norm3d} by constructing the estimates
\begin{align*}
\|\tilde{\V p}(u)\|_2^2=\sum_{n=1}^{N_0} \big|\langleÊ\tilde p_n, u \rangle\big|^2=\sum_{n=1}^{N_0} \left|\sum_{m=1}^{N_0} b_{m,n} \langleÊp_m, u \rangle\right|^2
 \le \|\M B\|_F^2\; \|\V p(u)\|_2^2\\
 \|{\V p}(u)\|^2_2 =\sum_{n=1}^{N_0} \big|\langleÊp_n, u \rangle\big|^2=\sum_{n=1}^{N_0} \left|\sum_{m=1}^{N_0} c_{m,n} \langleÊ\tilde p_m, u \rangle\right|^2 \le \|\M C\|_F^2\; \|\tilde{\V p}(u)\|^2_2.
\end{align*}


%
%
\end{proof}

\section{Link with classical results}

\subsection{Operator-based solution of differential equations}
\label{Sec:SODE}
The use of the regularized inverse operator $\Lop_\V \phi^{-1}$ has been proposed for the resolution
of stochastic partial differential equations of the form (see \cite{Unser2014, Fageot2013})
$$
\Lop s=w \quad \mbox{ s.t. } \quad \V \phi(s)=\V b,
$$
where $w \in \Spc S'(\R^d)$ is a realization of a $p$-admissible white-noise innovation process
and $\V b \in \R^{N_0}$ is a boundary-condition vector that may be deterministic (the typical choice being $\V b=\V 0$) or not.
Under the assumption that $\Lop_\V \phi^{-1\ast}: \Spc S(\R^d)\toC L_p(\R^d)$ with $p\ge1$, the solution is then given by $s=\Lop_\V \phi^{-1}w + \sum_{n=1}^{N_0} b_n p_n$.
The connection with the present work is that there are corresponding results available
on specific choices of $\V \phi$ that guarantee the continuity of $\Lop_\V \phi^{-1\ast}$
\cite[Chap 5]{Unser2014book}.
The scenario most studied in one dimension is $\Lop=\Dop$ with $(\phi_1,p_1)=(\delta,1)$, as it enables
the construction of the whole family of L\'evy processes \cite{Levy1954b}. These can be described as $s=\Dop^{-1}_{\delta}w$, where $w$ is a white L\'evy noise, which is compatible with the classical boundary condition $s(0)=\langle \delta, s \rangle=0$ \cite[Section 7.4, pp. 163-166]{Unser2014book}. Instead of the standard integrator $\Dop^{-1}$, which outputs the primitive of the function, the scheme uses the anti-derivative operator
$$
\Dop^{-1}_{\delta}\{\varphi\}(x)=\int_{0}^x \varphi(y)\dint y,
$$
whose adjoint $\Dop^{-1\ast}_{\delta}$ is continuous $\Spc S(\R) \toC \Spc R(\R)=\cap_{\alpha\in \Z} L_{\infty,\alpha}(\R)$ (the Fr\'echet space of rapidly decreasing functions)
\cite[Theorem 5.3, p. 100]{Unser2014book}.
Since 
$L_q(\R^d) \embedC \Spc R(\R)$ for all $q\le1$ and $\|\Dop^{-1\ast}_{\delta}\{\varphi\}\|_{L_q}\le \|\varphi\|_{L_1}$, we can readily extend the domain of continuity of the adjoint  pseudo-inverse to
$\Dop^{-1\ast}_{\delta}: L_1(\R) \toC L_q(\R^d)$. By taking 
$\Spc X'=L_p(\R)=\big(L_q(\R)\big)'$ with $p\ge1$, this then leads to 
the native Banach spaces
$$L_{p,\Dop}(\R)=\{f: \R \to \R: \|f\|_{p,\Dop}\eqdef \|\Dop f\|_{p} + |f(0)|<\infty\},$$
 which are Sobolev spaces of degree $1$. Theorem \ref{Theo:NativeSpace2} ensures that
 $\Spc S(\R) \embedC L_{p,\Dop}(\R)\embedC L_\infty(\R)$, which is consistent with the classical embedding properties of Sobolev spaces. In fact, the statement can be refined to
 $L_{p,\Dop}(\R)\embedC C_{\rm b}(\R)$ for any $p\ge 1$(see \cite{Unser2019b}).
 
 \subsection{Total variation and ${\rm BV}$}
 While the connection in Section \ref{Sec:SODE} is enlightening, it does not cover the case 
 $(\Lop,\Spc X')=\big(\Dop,\Spc M(\R)\big)$ (total variation) with $\Spc X=C_0(\R)$ because $\Dop^{-1\ast}_{\delta}\{\varphi\}$ does systematically present a discontinuity at the origin when $\langle \varphi,1\rangle\ne 0$, even though it is smooth everywhere else (see \cite[ Figure 5.1, p. 91]{Unser2014book}).
The problem is that
 $\delta \notin C_{0,\Dop}(\R)$. This can be fixed by selecting a more regular boundary functional (\ie any $\phi_1 \in L_1(\R)$ with $\langle \phi_1,1\rangle=1$) which then yields a corrected operator that is universal in the sense that $\Dop^{-1\ast}_{\phi_1}: \Spc S(\R) \toC \Spc S(\R)$.
It allows us to specify the proper native space
$$\Spc M_{\Dop}(\R)=\{f: \R \to \R \ \big| \ \| f\|_{\Spc M,\Dop}\eqdef \|\Dop f\|_{\Spc M} + |\langle \phi_1,f\rangle|<\infty\},$$
which extends ${\rm BV}(\R)$ (functions of bounded variations) slightly.
In the classical definition of ${\rm BV}(\R)$, the second term in the norm is replaced by $\|f\|_{1}$. This is more constraining as it makes the null space trivial by excluding constant signals.
In contrast with $L_{p,\Dop}(\R)$ including the limit scenario $p=1$,
the continuity of the members of $\Spc M_{\Dop}(\R)$ or ${\rm BV}(\R)$ is guaranteed only {\em almost everywhere}:
in other words, it can happen that the term $|f(0)|$ is not well-defined, which is the fundamental reason why it needs to be replaced by $|\langle \phi_1,f\rangle|$.

\subsection{Sobolev/Beppo-Levi spaces with $d=1$}

We have seen that the $N_0$th derivative operator
$\Lop=\Dop^{N_0}$ is spline-admissible with $\alpha=(N_0-1)$ and $\Spc N_{\Dop^{N_0}}=\{x^n\}_{n=0}^{N_0-1}$ (polynomials of degree $N_0-1$). 
Since we already known that $\Dop^{-1\ast}_{\delta}: L_1(\R)\toC L_p(\R)$, $p=1$ included, we can iterate the operator to construct an admissible pseudo-inverse of $\Dop^{m\ast}=(-1)^{m}\Dop^{m}$ as
$$
\Dop^{-m\ast}_{\V \phi}=\left(\Dop^{-1\ast}_{\delta}\right)^{m}: L_1(\R)\toC L_p(\R).
$$
with $N_0=m$ and $\phi_n=(-1)^{(n-1)}\delta^{(n-1)}$.
Indeed, since $f=\Dop^{-1}_{\delta}\{w\}$ imposes the boundary condition $f(0)=0$ and is left-invertible 
with $w=\Dop f$, 
$g=\left(\Dop^{-1}_{\delta}\right)^{m} \{w\}$ is invertible as well and such that
$g^{(n)}(0)=\langle (-1)^{(n-1)}\delta^{(n-1)}, g\rangle=0=\langle \phi_{n+1}, g\rangle$ for $n=0,\dots,(m-1)$
By observing that the underlying $\phi_n$ are biorthogonal to $\tilde p_n(x)=\frac{x^n}{n!}$, we can then safely define the
corresponding native spaces as
$$L_{p,\Dop^{m}}(\R)=\left\{f: \R \to \R: \|f\|_{p,\Dop^{m}}\eqdef \|\Dop^m f\|_{p} + \left(\sum_{n=0}^{m-1}\left|\frac{f^{(n)}(0)}{n!}\right|^2\right)^{1/2}<\infty\right\},$$
which, as expected, are Sobolev spaces of order $m$, albeit homogeneous extensions of the classical ones for they also includes the polynomials of degree less than $m$.
For $p=2$, we recover the typical kind of Beppo-Levi space \cite{Atteia1992} used to specify smoothing splines;
\ie the classical form of variational polynomial splines, which goes back to the pioneering works of Schoenberg and de Boor \cite{Schoenberg1964,deBoor1966}.


\subsection{Reproducing-kernel Hilbert spaces}
The best known examples of native spaces on $\R^d$ are RKHS \cite{Aronszajn1950,Berlinet2004,Wendland2005,Schaback2000}. They are included in the framework by taking $\Spc X=L_2(\R^d)$. In that case, the stability condition
$\Lop_{\V \phi}^{-1\ast}: L_{1,-\alpha}(\R^d) \toC L_2(\R^d)$ implies that $\Lop_{\V \phi}^{-1}: L_2(\R^d) \toC L_{\infty,\alpha}(\R^d)$.
This means that 
the canonical inverse operator $\Lop^{-1\ast}$ has the unique extension $\Lop^{-1\ast}=\Lop_\V \phi^{-1\ast}: \Spc U=\Lop^{\ast}\big(L_2(\R^d)\big)\toC L_2(\R^d)$
with $\|u\|_{2,\Lop}=\|\Lop^{-1\ast}u\|_{2}=\|\Lop_\V \phi^{-1\ast}u\|_{L_2}$ for any $u \in \Spc U$. This allows us to define the pair of self-adjoint operators $\Op A=(\Lop^{-1}\Lop^{-1\ast}): \Spc U \toC \Spc U'$
and $\Op A_\V \phi=(\Lop_\V \phi^{-1}\Lop_\V \phi^{-1\ast}): L_{2,\Lop}(\R^d) \toC L'_{2,\Lop}(\R^d)$.
Since $\Spc S_{\V p^\perp}(\R^d)\embedD \Spc U\embedC \Spc X_\Lop$ (by Corollary \ref{Corol:U}), we have that 
\begin{align}
\forall \varphi \in \Spc S_{\V p}(\R^d) \backslash \{ 0\}: \|\varphi\|_{2,\Lop}&=\langle \Lop^{-1\ast}\varphi,\Lop^{-1\ast}\varphi \rangle \nonumber\\
&=\langle \Op A \varphi, \varphi \rangle=\langle \Op A_\V \phi \varphi, \varphi \rangle>0.\label{Eq:Posdef}
\end{align}
This expresses the (strict) $\V p$-conditional positive definiteness of $\Op A$ (resp. $\Op A_\V \phi$)
which can also be identified as the inverse (resp. the pseudo-inverse) of $(\Lop^\ast\Lop)$.
In the particular case where $\V p$ is a basis of the polynomials of degree $n_0$, Condition \eqref{Eq:Posdef} is equivalent to the notion of $n_0$th-order conditional positive definiteness used in approximation theory \cite{Micchelli1986,Wendland2005}.
It is a classical hypothesis in the theory of (semi-)RKHS and is also necessary for our construction.
Classically, the strict ($n_0$th-order conditional) positive definiteness of $\Op A$ (or of its underlying kernel) it known to be sufficient to yield a 
(semi-)RKHS. This is not quite the case here because we also want the embedding $\Spc S(\R^d) \embedC L'_{2,\Lop}(\R^d) \embedC
L_{\infty,\alpha}(\R^d)$ that controls the growth of the members of the native space. The latter calls for the continuity of $\Lop_{\V \phi}^{-1\ast}: L_{1,-\alpha}(\R^d) \toC L_2(\R^d)$ ($L_2$-stability) or, equivalently,
of $\Op A_{\V \phi} 
: L_{1,-\alpha}(\R^d) \toC L_{\infty,\alpha}(\R^d)$.
\subsection{Connections with kernel methods and splines}
In \cite{Unser2019b}, we shall identify a simple condition on the kernel of $\Op A=(\Lop\Lop^\ast)^{-1}$ that ensures that the stability requirement in Definition \ref{Def:Hypotheses} for $\Spc X=L_2(\R^d)$ is met. This will enable us to
prove that the combination of spline admissibility in Definition \ref{Def:splineadmis3} and the classical (conditional-)positivity requirement \eqref{Eq:Posdef} are necessary and sufficient for the native space of $(\Spc X',\Lop)$ with $\Spc X=\Spc X'=L_2(\R^d)$
to be a RKHS, with the property that $\Spc S(\R^d) \embedC L'_{2,\Lop}(\R^d) \embedC C_{{\rm b},\alpha}(\R^d)$.
We shall further the argument by reformulating the primary results of the present paper in terms of kernels, rather than operators. This will provide us with explicit criteria for checking that the compatibility conditions in Definition \ref{Def:Hypotheses} are met for a broad variety of primary spaces $\Spc X'$. 
We shall also devote a particular attention to the case $\Spc X'=\Spc M(\R^d)$, which is central to the theory of $\Lop$-splines \cite{Unser2017}.
Examples of applications of our native Banach-space formalism, including the resolution of variational inverse problems and the derivation of representer theorems, will be presented in \cite{Fageot2019a}.

%

\pagebreak
\appendix

\section*{Appendix A: Topological embeddings}
\label{App:Embeddings}

The notion of embedding for topological vector spaces comes in four gradation: inclusion as a set (symbolized by $\Spc X\subseteq\Spc Y$),
continuous embedding ($\Spc X \embedC \Spc Y$), isometric embedding ($\Spc X \embedIso \Spc Y$), and, finally, continuous and dense embedding ($\Spc X\embedD\Spc Y$).

\begin{definition}[Continuous embedding]
\label{Def:embed}
Let $\Spc X$ and $\Spc Y$ be two locally convex topological vector spaces where $\Spc X\subseteq \Spc Y$ (as a set).
$\Spc X$ is said to be continuously embedded in $\Spc Y$, which is denoted by $\Spc X \embedC \Spc Y$, if the inclusion/identity map
$
\Op I: \Spc X \to \Spc Y: x \mapsto x
$
is continuous. 
\end{definition}
In particular, if $\Spc X$ and $\Spc Y$ are two Banach spaces, then the definition can be restated as:
for all $x \in \Spc X$, $\Op I\{x\}=x \in \Spc Y$ with $\|x\|_{\Spc Y} \le C_0 \|x\|_{\Spc X}$ for some constant $C_0>0$.
If, in addition,
$\|x\|_{\Spc X}=\|x\|_{\Spc Y}$ for all $x \in \Spc X \subseteq \Spc Y$, then the embedding is isometric, a property that is denoted by $\Spc X \embedIso \Spc Y$.
For instance, a classical result is that any Banach space $\Spc X$ is isometrically embedded in its bidual; \ie $\Spc X \embedIso \Spc X''$. In fact, we have that $\Spc X=\Spc X''$ (meaning that the two spaces are isometrically isomorphic) if and only if $\Spc X$ is reflexive.

An example of such embeddings that is relevant to this paper is 
$$\Spc S(\R^d) \embedC L_{1,-\alpha}(\R^d) \embedIso \big(L_{1,-\alpha}(\R^d)\big)''=(L_{\infty,\alpha}(\R^d)'.
$$

So far we have emphasized the property of continuity, but there are instances such as $\Spc S(\R^d) \embedC L_{1,-\alpha}(\R^d)$ that are more powerful because the embedding also happens to be dense; \ie $\Spc X \embedD \Spc Y$, where $\Spc Y$ can be specified as the completion of $\Spc X$ for $\|\cdot\|_{\Spc Y}$.

\begin{definition}[Dense embedding]
\label{Def:Dense}
Let  $\Spc X$ be a linear subspace of a locally convex topological vector space $\Spc Y$. Then, $\Spc X$ is said to be dense in $\Spc Y$ if it has the ability to separate distinct elements of the dual space $\Spc Y'$; that is, if, for any $y' \in \Spc Y'$,
$$\langle y',x\rangle_{\Spc Y' \times \Spc Y}=0 \mbox{ for all } x\in \Spc X\subseteq \Spc Y \ \ \Leftrightarrow \ \ y'=0.$$
\end{definition}

In the case where $\Spc Y$ is a Banach space, the denseness of $\Spc X$ has another equivalent formulation:
for any $y\in \Spc Y$ and $\epsilon>0$, there exists some $x_\epsilon \in \Spc X$ such that
$\|y-x_\epsilon\|_{\Spc Y}< \epsilon$, which means that $\Spc X$ is rich enough to represent any element
of $\Spc Y$ with an arbitrary degree of precision.

\begin{theorem}[Dual embedding] 
\label{Theo:Dense}
 Let  $\Spc X$ and $\Spc Y$ be two locally convex topological vector spaces
 such that  $\Spc X \embedD\Spc Y$, where the embedding is continuous and dense. Then,
 $\Spc Y' \embedC \Spc X'$. Moreover,  if $\Spc X$ is reflexive, then $\Spc Y'\embedD \Spc X'$.
 
Likewise, in the case of Banach spaces, $\Spc X \embedIso\Spc Y$ implies that $\Spc Y' \embedC\Spc X'$ with bounding constant one (preservation of norm). Moreover, if there exists a topological vector space $\Spc S$ such that 
$\Spc S\embedD \Spc X$ and $\Spc S\embedD \Spc Y$,
then $\Spc Y'\embedIso \Spc X'$.

\end{theorem}
\begin{proof} Since $\Spc X\embedC\Spc Y$, the linear functionals that are continuous on $\Spc Y$ are also continuous on $\Spc X$ so that $\Spc X'\subseteq\Spc Y'$ (as a set). Moreover, the continuity of the identity/inclusion map $\Op I: \Spc X \to \Spc Y$ implies the continuity of its adjoint 
$\Op I^\ast: \Spc Y' \to \Spc X'$, defined as 
\begin{align}
\langle \Op I^\ast \{y'\}, x\rangle_{\Spc X' \times \Spc X}=\langle y', \Op I\{x\}\rangle_{\Spc Y' \times \Spc Y}=\langle y', x\rangle_{\Spc Y' \times \Spc Y}
\label{Eq:iast}
\end{align}
for all $y' \in \Spc Y', x \in \Spc X$. Finally, the denseness of $\Spc X$ in $\Spc Y$ ensures that $\Op I^\ast$ is the correct inclusion map with $\Op I^\ast \{y'\}=y'$,
 which proves that $\Spc Y' \embedC \Spc X'$. 

\item {\em Second part by contradiction}:
 Suppose that $\Spc Y'$ is not dense in $\Spc X'$. Then, there is an $x_0^{\prime\prime}\in \Spc X^{\prime\prime}$ that is not identically zero such that 
$\langle x_0^{\prime\prime}, y^{\prime}\rangle_{\Spc X^{\prime\prime}\times\Spc X'}=0$
for all $y^{\prime} \in \Spc Y'\subseteq \Spc X'$ (contrapositive of the statement in Definition  \ref{Def:Dense}). Moreover, due to the reflexivity of $\Spc X$, there is a corresponding $x_0 \in \Spc X \embedC \Spc Y$ such that $\Op B\{x_0\}=x_0^{\prime\prime}$, where $\Op B: \Spc X\to \Spc X''$ is the canonical bijective mapping for a reflexive space to its  bidual. Therefore,
\begin{align*}
	0
	&=
	\langle x_0^{\prime\prime},y^{\prime}\rangle_{\Spc X^{\prime\prime}\times\Spc X'}=
	\langle \Op B\{x_0\},y^{\prime}\rangle_{\Spc X^{\prime\prime}\times\Spc X'}\nonumber \\
	&=
	\langle y^{\prime},x_0\rangle_{\Spc X'\times \Spc X}=\langle y^{\prime},x_0\rangle_{\Spc Y'\times \Spc Y},
\end{align*}
for all $y' \in \Spc Y'$. Since the topological spaces $\Spc Y'$ and $\Spc Y$ form a dual pair, the identity $\langle y^{\prime},x_0\rangle_{\Spc Y'\times \Spc Y}=0$ implies that $x_0=0$, which is a contradiction.

\item {\em Banach Isometries}: From the definition of the dual norm and the property that
$\|\varphi\|_{\Spc X}=\|\varphi\|_{\Spc Y}$ when $\varphi \in \Spc X$, for any $y' \in \Spc Y'\subseteq \Spc X'$, we have that
\begin{align}
\label{Eq:Normembeds}
\|y'\|_{\Spc Y'}=\sup_{\varphi \in \Spc Y\backslash \{0\}} \frac{\langle  y,\varphi\rangle }{\|\varphi\|_{\Spc Y}} \ge \sup_{\varphi \in \Spc X\backslash \{0\}} \frac{\langle  y,\varphi\rangle }{\|\varphi\|_{\Spc Y}}=\sup_{\varphi \in \Spc X\backslash \{0\}} \frac{\langle  y,\varphi\rangle }{\|\varphi\|_{\Spc X}}
=\|y'\|_{\Spc X'}
\end{align}
where the inequality results from the search space $\Spc X$ on the right being a subset of $\Spc Y$; hence, 
$\Spc X' \embedC \Spc Y'$. In the second scenario, the two search spaces in \eqref{Eq:Normembeds} can be replaced by the dense subspace $\Spc S$, which then yields an equality. However, this setting also implies that $\Spc X=\Spc Y$ as both spaces are the unique completion of $\Spc S$ in the $\|\cdot\|_{\Spc X}=\|\cdot\|_{\Spc Y}$-norm.

\end{proof}

As demonstration of usage, we 
now prove the following remarkable result in Schwartz' theory of distributions:

\begin{proposition}
\label{Prop:Sdense}
$\Spc S(\R^d)$ is continuously and densely embedded in $\Spc S'(\R^d)$; \ie $\Spc S(\R^d) \embedD \Spc S'(\R^d)$.
\end{proposition}
\begin{proof} For any $\phi \in \Spc S(\R^d)$, the map $\varphi \mapsto \langle \phi, \varphi\rangle=\int_{\R^d} \phi(\V x)\varphi(\V x) \dint \V x$ specifies a continuous linear functional over $\Spc S(\R^d)$, which already shows that $\Spc S(\R^d)\subseteq \Spc S'(\R^d)$ (as a set).
To prove that the embedding is continuous, we invoke the continuity of the identity operator $\Op I: \Spc S(\R^d) \to \Spc S'(\R^d)$, which is obvious from the underlying topology.
As for the denseness property, the relevant annihilator space is
$$
\Spc S^\perp=\{\varphi \in \big( \Spc S'(\R^d)\big)'=\Spc S(\R^d): \langle \varphi, \phi\rangle=0 \mbox{ for all } \phi \in \Spc S(\R^d)\}.
$$
In particular, if $\varphi \in \Spc S^\perp\subseteq \Spc S(\R^d)$, then $\langle \varphi, \varphi\rangle=0 \Leftrightarrow \|\varphi\|_{L_2}=0 \Leftrightarrow \varphi=0$,
which proves that $ \Spc S^\perp=\{0\}$.
\end{proof}

In practice, it is often easier to prove that an embedding is continuous than establishing its denseness.
Fortunately, it is possible to transfer such properties by taking advantage
of functional hierarchies.
\begin{proposition}[Hierarchy of dense embeddings]
\label{Prop:Hierarchy}
Let $\Spc X$, $\Spc Y$, and $\Spc Z$ be three locally convex topological vector spaces such that $\Spc X \embedC \Spc Y \embedC \Spc Z$
(continuous embeddings). Then, we have the following implications (dense embeddings).
\begin{enumerate}
\item $\Spc X \embedD\Spc Y$ and $\Spc Y \embedD \Spc Z$ $\ \RightarrowÊ\ $ $\Spc X\embedD\Spc Z$ \quad 
\item $\Spc X \embedD\Spc Z$ $\ \RightarrowÊ\ $ $\Spc Y \embedD\Spc Z$.
\end{enumerate}
\end{proposition}
\begin{proof}
\item Statement 1: $\Spc X \subseteq \Spc Y \subseteq \Spc Z$ as sets.
Since the closure of $\Spc Y$ in the topology of $\Spc Z$ is $\Spc Z$,
the closure of $\Spc X$ in the topology of $\Spc Z$ must also be $\Spc Z$.

\item Statement 2: The annihilators of $\Spc X\subseteq \Spc Z$ and  $\Spc Y\subseteq \Spc Z$ in $\Spc Z'$ are
\begin{align*}
\Spc X^\perp =\{u\in \Spc Z': \langle u,x\rangle_{\Spc Z' \times \Spc Z}=0 \mbox{ for all } x\in \Spc X\}\\
\Spc Y^\perp =\{u\in \Spc Z': \langle u,y\rangle_{\Spc Z' \times \Spc Z}=0 \mbox{ for all } y\in \Spc Y\}
\end{align*}
with the property that $\Spc Y^\perp \subseteq \Spc X^\perp$ (as a set) from the definition. Hence,
$\Spc X^\perp=\{0\} \Rightarrow \Spc Y^\perp=\{0\}$.
%
%
\end{proof}
\section*{Appendix B: Direct-sum topology}
\begin{definition}
Let $\Spc U$ and $\Spc V$ be two subspaces of a linear space $\Spc W$.
Then, the sum space is $\Spc U+\Spc V=\{f=u+v: (u,v) \in \Spc U \times \Spc V\}\subset \Spc W$.
The sum is called direct and is notated $\Spc U \oplus \Spc V$ if
$\Spc U \cap \Spc V=\{0\}$. 
\end{definition}
If $\ \Spc U$ and $\Spc V$ in the above definition are normed with
respective norms $\|\cdot\|_{\Spc U}$ and $\|\cdot\|_{\Spc V}$, then
$\|(\|u\|_{\Spc U},\|v\|_{\Spc V})\|_p$ is a norm for both $\Spc U \times \Spc V$ and
 $\Spc U \oplus \Spc V$. The choice of the exponent $p\ge1$ in the composite norm is flexible since all finite-dimensional norms are equivalent.
One also defines the corresponding linear projection operators $\Proj_{\Spc U}: (\Spc U \oplus \Spc V) \to \Spc U$ and $\Proj_{\Spc V}: (\Spc U \oplus \Spc V) \to\Spc V$. These are such that, for any $(u,v)\in \Spc U \times \Spc V$,
\begin{align}
\Proj_{\Spc U}\{u+v\}=u,\\
\Proj_{\Spc V}\{u+v\}=v.
\end{align}
In summary, any $f \in  \Spc U \oplus \Spc V$ has a unique decomposition as 
$f=u+v$ with $\Proj_{\Spc U}\{f\}=u \in \Spc U$ and $\Proj_{\Spc V}\{f\}=v \in \Spc V$, while
$\|f\|_{\Spc U \oplus \Spc V}= \|(\|\Proj_{\Spc U}\{f\}\|_{\Spc U} ,\|\Proj_{\Spc V}\{f\}\|_{\Spc V})\|_p$. 

The concept is also applicable to Banach spaces, which are often specified explicitly as the completion of some normed space. This normed space is called a pre-Banach space when it is not yet completed.

\begin{theorem}[Completion of a direct sum space]
\label{Theo:DirectSum}
Let $\Spc U_{\rm pre} \oplus \Spc V_{\rm pre}$ be the direct sum of two normed spaces
$(\Spc U_{\rm pre},\|\cdot\|_{\Spc U})$ and $(\Spc V_{\rm pre},\|\cdot\|_{\Spc V})$.
Then, $\overline{\Spc U_{\rm pre} \oplus \Spc V_{\rm pre}}=
\Spc U \oplus \Spc V$, where $\Spc U=\overline{(\Spc U_{\rm pre},\|\cdot\|_{\Spc U})}$ and $\Spc V=\overline{(\Spc V_{\rm pre},\|\cdot\|_{\Spc V})} $ are the Banach spaces associated with the $\|\cdot\|_{\Spc U}$-norm and $\|\cdot\|_{\Spc V}$-norm, respectively.
Moreover, the resulting projection operators 
$\Proj_{\Spc U}: \Spc U \oplus \Spc V \toC \Spc U$ with $\|\Proj_{\Spc U}\|=1$ and $\Proj_{\Spc V}: \Spc U \oplus \Spc V \toC \Spc V$ with $\|\Proj_{\Spc V}\|=1$ are the unique continuous extensions
of $\Proj_{\Spc U_{\rm pre}}: \Spc U_{\rm pre} \oplus \Spc V_{\rm pre} \to \Spc U$ and $\Proj_{\Spc V_{\rm pre}}: \Spc U_{\rm pre} \oplus \Spc V_{\rm pre} \to \Spc V$, respectively.
\end{theorem}

\begin{proof}
Let $\Spc G_{\rm pre}=\Spc U_{\rm pre} \oplus \Spc V_{\rm pre}$, whose completion with respect to the
direct-sum norm is the Banach space $\Spc G$.
We then
pick some Cauchy sequence $(g_i)$ in $\Spc G_{\rm pre}$ that
converges to $g\in \Spc G$. 
Because the projectors $\Proj_{\Spc U_{\rm pre}}: \Spc P_{\rm pre} \to \Spc U_{\rm pre}$ and $\Proj_{\Spc V_{\rm pre}}: \Spc P_{\rm pre} \to \Spc V_{\rm pre}$ are contractive maps, the transformed sequences $(u_i)=(\Proj_{\Spc U_{\rm pre}} \{g_i\})$ and $(v_i)=(\Proj_{\Spc V_{\rm pre}} \{g_i\})$ are Cauchy in
$\Spc U_{\rm pre}$ and $\Spc V_{\rm pre}$, respectively, and converge to some limits
$u=\lim_{i\to\infty}u_i \in \Spc U \eqdef \overline{\Spc U}_{\rm pre}$ and $v=\lim_{i\to\infty}v_i \in \Spc V \eqdef \overline{\Spc V}_{\rm pre}$. Hence, $u+v \in \Spc U + \Spc V$.
Conversely, because of the direct sum property, the sum of any two Cauchy sequences in $\Spc U_{\rm pre}$ and $\Spc V_{\rm pre}$
is Cauchy in $\Spc G_{\rm pre}$ converges to $\lim_{i\to\infty}(u_i+v_i)=
\lim_{i\to\infty}u_i+\lim_{i\to\infty} v_i=u+v\in \overline{\Spc G}_{\rm pre}=\Spc G$, which shows that $\Spc U+\Spc V \subseteq \Spc G$. Hence, we conclude that $\Spc G=\Spc U+\Spc V$.

Since $\Spc U_{\rm pre}\embedIso \Spc U$ and $\Proj_{\Spc U_{\rm pre}}$ is a projector, it is bounded 
$\Spc G_{\rm pre} \to \Spc U$ with $\|\Proj_{\Spc U_{\rm pre}}\|=1$. By the B.L.T. theorem, it therefore admits a unique continuous extension
$\Proj_{\Spc U}: \overline{\Spc G}_{\rm pre}=\Spc G \toC \Spc U$ with $\|\Proj_{\Spc U}\|=1$, whose explicit definition is
\begin{align*}
\Proj_{\Spc U}\{g\}\eqdef \lim_{i\to\infty}\Proj_{\Spc U_{\rm pre}}\{g_i \}
\end{align*}
where $(g_i)$ is any Cauchy sequence in $\Spc G_{\rm pre}$ that converges to
$g \in \Spc G$. The same holds true for $\Proj_{\Spc V}: \Spc G \toC \Spc V$,  which is such that
\begin{align*}
\Proj_{\Spc V}\{g\}\eqdef \lim_{i\to\infty}\Proj_{\Spc V_{\rm pre}}\{g_i\}=\lim_{i\to\infty} v_i=v.
\end{align*}
We then use these extended operators to show that the sum $\Spc U + \Spc V$ is direct. Specifically, by invoking the basic properties of $\Proj_{\Spc U_{\rm pre}}$, we get
\begin{align*}
\Proj_{\Spc U}\{g\} &= \lim_{i\to\infty}\Proj_{\Spc U_{\rm pre}}\{\Proj_{\Spc U_{\rm pre}}\{g_i\}  + \Proj_{\Spc V_{\rm pre}}\{g_i  \}\}\\
&= \lim_{i\to\infty}\left(\Proj_{\Spc U_{\rm pre}}\{u_i\} +\Proj_{\Spc U_{\rm pre}}\{v_i \}\right)= \lim_{i\to\infty}\left(u_i + 0\right)=u \in \Spc U,
\end{align*}
which is equivalent to $\Proj_{\Spc U}\{u+v\}=u$ for any $(u,v)\in\Spc U\times\Spc V$.
Correspondingly, we also obtain that $\Proj_{\Spc V}g=\lim_{i\to\infty}(0+v_i)=v=(g-u) \in \Spc V$ with $\Proj_{\Spc V}\{u\}=0$ and $\Proj_{\Spc V}\{v\}=v$,
which proves that $\Spc U\cap\Spc V=\{0\}$.

\end{proof}
The final element is the identification of the dual space which, as expected, also has a direct-sum structure (see \cite[Theorem 1.10.13]{Megginson1998}), albeit with a suitable adaptation of the composite norm.
\begin{proposition}
\label{Prop:dualDirectSum}
Let $(\Spc U,\|\cdot\|_{\Spc U})$ and $(\Spc V,\|\cdot\|_{\Spc V})$ be two complementary Banach subspaces of $\Spc W$ and $\Spc U \oplus \Spc V$ the corresponding direct-sum space equipped with the composite norm
$\|(\|u\|_{\Spc U},\|v\|_{\Spc V})\|_p$.
Then, the continuous dual of $\Spc U \oplus \Spc V$ is the the direct-sum Banach space
$\Spc U' \oplus \Spc V'$ equipped with the dual composite norm
$\|u'+v'\|_{\Spc U' \oplus\Spc V'}=\|(\|u'\|_{\Spc U'},\|v'\|_{\Spc V'})\|_q$ where $q=\frac{p}{p-1}$ is the conjugate exponent of $p\ge1$.
\end{proposition}
\subsection*{Acknowlegdments}
The research was partially supported by the Swiss National Science
Foundation under Grant 200020-162343.
The authors are thankful to John-Paul Ward and Shayan Aziznejad for helpful discussions.


\bibliographystyle{abbrv}
%
%

\bibliography{/GoogleDrive/Bibliography/BibTex_files/Unser}
\end{document}